\documentclass[10pt]{article}

\usepackage{amsmath, amsthm, amssymb}
\usepackage{caption, graphicx, paralist}
\usepackage{authblk}
\usepackage[text={6.5in,9.25in},centering]{geometry}
\usepackage[sort, numbers, compress]{natbib}

\makeatletter\@addtoreset{equation}{section}\makeatother

\newtheorem{thm}{Theorem}[section]
\newtheorem{stat}[thm]{Statement}
\newtheorem{lem}[thm]{Lemma}

\newtheorem{hyp}{Hypothesis}
\newtheorem{defn}[thm]{Definition}

\newtheoremstyle{named}{}{}{\itshape}{}{\bfseries}{.}{.5em}{\thmnote{#3's }#1}
\theoremstyle{named}

\theoremstyle{definition}

\newenvironment{Acknowledgment}%
{\begin{trivlist}\item[]\textbf{Acknowledgments }}{\end{trivlist}}

\newcommand{\R}{\mathbb{R}}
\newcommand{\C}{\mathbb{C}}
\newcommand{\Z}{\mathbb{Z}}

\newcommand{\rmi}{\mathrm{i}}
\newcommand{\rme}{\mathrm{e}}

\def\Im{\mathop\mathrm{Im}\nolimits}    % imaginary part

\begin{document}

%\title{Localized synchronous patterns in weakly coupled bistable oscillators}
\title{Localized synchrony patterns in weakly coupled bistable oscillator systems}

\author[1]{Erik Bergland}
\author[2]{Jason Bramburger}
\author[3]{Bj\"orn Sandstede}
\affil[1]{\small One Health Trust, Washington, DC, USA}
\affil[2]{\small Department of Mathematics and Statistics, Concordia University, Montr\'eal, QC, Canada}
\affil[3]{\small Division of Applied Mathematics, Brown University, Providence, RI, USA}

\date{}

\maketitle

\begin{abstract}
Motivated by numerical continuation studies of coupled mechanical oscillators, we investigate branches of localized time-periodic solutions of one-dimensional chains of coupled oscillators. We focus on Ginzburg--Landau equations with nonlinearities of Lambda-Omega type and establish the existence of localized synchrony patterns in the case of weak coupling and weak-amplitude dependence of the oscillator periods. Depending on the coupling, localized synchrony patterns lie on a discrete stack of isola branches or on a single connected snaking branch.
\end{abstract}

%%%%%%%%%%%%%%%%%%%%%%%%%%%%%%%%%%%%%%%%%%%%%%%%%%%%%%%%%%%%%%%%%%%%%%%%%%%%
%%%%%%%%%%%%%%%%%%%%%%%%%%%%%%%%%%%%%%%%%%%%%%%%%%%%%%%%%%%%%%%%%%%%%%%%%%%%

\section{Introduction}

Coupled oscillators arise in many natural and engineering systems, and their dynamics has been the focus of much work, including \cite{kuramoto1984chemical, ermentrout1991multiple, ermentrout1990oscillator, brown2003globally, heagy1994synchronous, hoppensteadt1997weakly}. A typical scenario starts with identical oscillators that, when uncoupled, evolve to a stable periodic orbit of a dynamical system. When these oscillators are coupled, synchronous oscillations can arise where the phases of the different oscillators lock to generate a globally periodic motion. Prominent examples of the emergence of synchronized oscillations are neuronal activity patterns in the brain, the firing patterns of fireflies, and Huygens synchronization of clocks.

\begin{figure}
\centering
\includegraphics[scale=0.8]{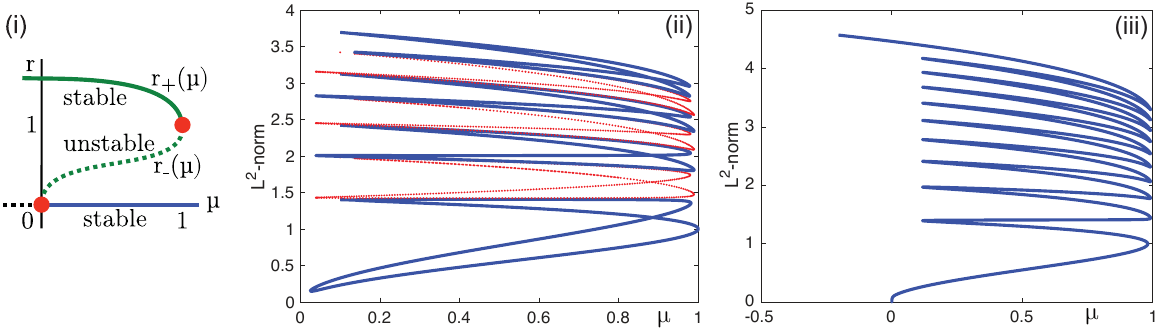}
\caption{Panel~(i) illustrates our assumption that there are two oscillatory states ($r_-(\mu)$ is unstable, and $r_+(\mu)$ is stable) and one stationary stable state $r=0$ inside the bistability region $0<\mu<1$. The stationary state $r=0$ exhibits a Hopf bifurcation at $\mu=0$, while the two oscillatory states merge in a fold bifurcation at $\mu=1$. Panels~(ii) and~(iii) show numerical continuation results of, respectively, isola and snaking branches of \eqref{i:model} with conservative ($c=\rmi$) and dissipative ($c=1$) coupling for $N=10$ nodes with coupling strength $\varepsilon=0.01$ and the nonlinearity given in \eqref{e:exnonl}.}
\label{f:synchrony}
\end{figure}

Our focus will be on weakly coupled oscillator systems where each oscillator admits bistable dynamics with a stable oscillation (represented by a stable periodic orbit), an unstable oscillation, and a stable rest state that coexist as indicated in Figure~\ref{f:synchrony}(i). We are interested in the emergence and the parameter-dependence of solutions where a set of neighboring nodes exhibit synchronized oscillations, while all other nodes remain at the rest state: we refer to such solutions as localized synchrony patterns. Understanding localized synchrony patterns is relevant for localized breathers of discrete nonlinear Schr\"odinger equations \cite{shi2015existence, parker2020existence, flach1998discrete, cuevas2019discrete, parker2021standing} and localized oscillations in chains of coupled mechanical oscillators \cite{clerc2018chimera, fontanela2018dark, niedergesass2021experimental, papangelo2018multiple, vakakis2001normal, shiroky2020nucleation, chiu2001synchronization, papangelo2017snaking}.

Complicated branches of localized synchrony patterns were found in \cite{papangelo2017snaking} in a one-mode Fourier approximation arising in the harmonic-balance method (HBM) approximation of a system of weakly-coupled mechanical oscillators with linear springs and nonlinear dampers. The HBM model of this mechanical system is of the form
\begin{equation}\label{i:model}
\dot Z_n = f(|Z_n|,\mu) Z_n + \varepsilon c (Z_{n+1}-2Z_n+Z_{n-1}),
\qquad Z_n\in\C, \qquad n\in\Z
\end{equation}
for small coupling strength $0<\varepsilon\ll1$ and an appropriate complex-valued nonlinearity $f$ that depends on a parameter $\mu$. For the mechanical system considered in \cite{papangelo2017snaking}, the coupling constant $c$ is given by $c=\rmi$ and represents conservative coupling. For dissipative coupling with coupling constant $c=1$, the system (\ref{i:model}) can be interpreted as a generalized Ginzburg--Landau (or Lambda-Omega) system that describes, for instance, spiral waves in weakly coupled excitable media; see \cite{aguareles2016asymptotic, paullet1994existence, cohen1978rotating, troy2022logarithmic} for the spatially continuous setting and \cite{bramburger2020stable, paullet1998spiral, bramburger2019rotating} for systems posed on lattices. In Figure~\ref{f:synchrony}, we demonstrate through numerical simulations that localized synchrony patterns emerge in the model \eqref{i:model} that lie on discrete stacks of isolas for $c=\rmi$ and on connected snaking branches for $c=1$. The structure of the localized synchrony patterns along these branches is illustrated in Figure~\ref{f:patterns}.

\begin{figure}
\centering
\includegraphics[scale=1]{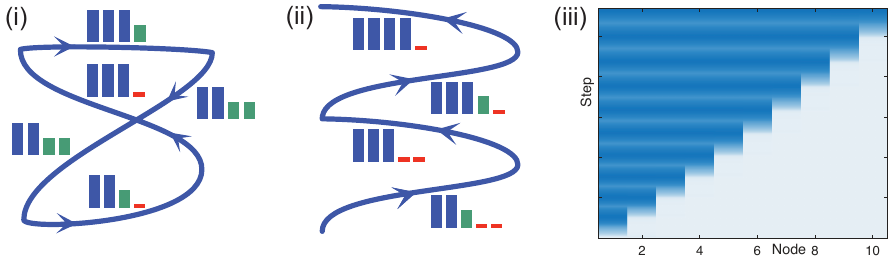}
\caption{Panels~(i) and~(ii) show how the localized synchrony patterns change along the isola and snaking branches, where the tall (blue), medium (green), and low (red) rectangles represent $r_+$, $r_-$, and $0$, respectively. Panel~(iii) contains numerical continuation results that indicate how nodes are recruited along the snaking branch, starting with a single small-amplitude oscillator and ending with a fully oscillatory pattern.}
\label{f:patterns}
\end{figure}

In this paper, we explore the following overarching questions: under weak coupling, do oscillator chains exhibit localized synchrony patterns across the bistability region and, if so, what is the geometry of the resulting branches. Our results for the system \eqref{i:model} with weak coupling $0<\varepsilon\ll1$ can be summarized informally as follows:
\begin{compactenum}[(1)]
\item For dissipative coupling ($c=1$) and with $\Im f(r,\mu):=\omega_0(\mu)$, localized synchrony patterns exist and lie on connected snaking branches [rigorous].
\item For conservative coupling ($c=\rmi$) and with $\Im f(r,\mu):=\omega_0(\mu)$, localized synchrony patterns exist and lie on discrete stacks of isolas [not rigorous].
\item Localized synchrony patterns across the bistability region cannot occur when the frequencies of stable and unstable oscillations are not sufficiently close to each other compared to the coupling strength $\varepsilon$ for $0<\varepsilon\ll1$ [rigorous].
\end{compactenum}
We refer to \S\ref{s:2} for the formal statements of these results for \eqref{i:model}. In particular, our results for \eqref{i:model} suggest that the isolas found numerically for the HBM approximation of the mechanical oscillator system considered in \cite{papangelo2017snaking} may not persist for the full mechanical system, and we explore this in more detail in \S\ref{s:2.5}.

Our results indicate that the specific nature of the coupling strongly impacts the geometry of branches of localized synchrony patterns: discrete stacks of isolas and connected snaking branches are possible for the same bistable dynamics of uncoupled oscillator nodes, and which of these cases occurs depends only on the structure of the linear coupling across nodes. We will comment more in \S\ref{s:5} on how general these results are beyond the system \eqref{i:model}.

The remainder of this paper is organized as follows. In \S\ref{s:2}, we present our main results and the application to the full mechanical system considered in \cite{papangelo2017snaking}. The proofs of our main results are contained in \S\ref{s:3}-\ref{s:4}, and we end in \S\ref{s:5} with conclusions and a discussion of open problems.

%%%%%%%%%%%%%%%%%%%%%%%%%%%%%%%%%%%%%%%%%%%%%%%%%%%%%%%%%%%%%%%%%%%%%%%%%%%%
%%%%%%%%%%%%%%%%%%%%%%%%%%%%%%%%%%%%%%%%%%%%%%%%%%%%%%%%%%%%%%%%%%%%%%%%%%%%

\section{Results: Localized synchrony patterns of coupled oscillators}\label{s:2}

In this section, we describe our main results. We introduce the precise setup in \S\ref{s:2.1} and outline our results for snaking branches under dissipative coupling and isola branches under conservative coupling in \S\ref{s:2.2} and \S\ref{s:2.3}, respectively. In \S\ref{s:2.4}, we argue why localized synchrony patterns cannot lie on isola or snaking branches under weak coupling when the stable and unstable oscillations have different frequencies. We also apply our results to both the full model and the harmonic-balance approximation of the mechanical oscillator chain investigated in \cite{papangelo2017snaking}.

%%%%%%%%%%%%%%%%%%%%%%%%%%%%%%%%%%%%%%%%%%%%%%%%%%%%%%%%%%%%%%%%%%%%%%%%%%%%

\subsection{Model system}\label{s:2.1}

We consider the system
\begin{equation}\label{e:ode}
\dot Z_n = f(|Z_n|,\mu, \varepsilon) Z_n + \varepsilon c (Z_{n+1}-2Z_n+Z_{n-1}),
\qquad Z_n\in\C, \qquad n\in\Z
\end{equation}
of coupled oscillators for a fixed coupling constant $c\in\C$ with $|c|=1$. The system \eqref{e:ode} is gauge invariant: if $(Z_n(t))_{n\in\Z}$ is a solution to \eqref{e:ode}, then so is $(\rme^{\rmi\alpha}Z_n(t))_{n\in\Z}$ for each fixed $\alpha\in\mathbb{R}$. Gauge invariance allows us to seek periodic solutions as relative equilibria that are of the form
\[
Z_n(t) = \rme^{\rmi\rho t} z_n, \qquad n\in\Z
\]
for time-independent complex amplitudes $z_n$, where $\rho\in\mathbb{R}$ is the temporal frequency of the periodic orbit. The amplitudes $z_n$ then satisfy the algebraic system
\begin{equation}\label{e:alg}
\rmi\rho z_n = f(|z_n|,\mu,\varepsilon) z_n + \varepsilon c (z_{n+1}-2z_n+z_{n-1}), \quad n\in\Z.
\end{equation}
We say that a solution $(z_n)_{n\in\Z}$ of \eqref{e:alg} is an \textbf{on-site} solution if the solution is invariant under the reflection across $n=1$ so that $z_n=z_{2-n}$ for all $n$, while a solution is said to be an \textbf{off-site} solution if it is invariant under the reflection across $n=\frac12$ so that $z_n=z_{1-n}$ for all $n$. We can find on- and off-site solutions by restricting the index $n$ to $n\geq1$ and including the boundary conditions $z_0=z_2$ for on-site and $z_0=z_1$ for off-site solutions. We write $z_n=r_n\rme^{\rmi\theta_n}$ with $r_n,\theta_n\in\mathbb{R}$ for each $n$. Gauge invariance implies that solutions depend only on the amplitudes $r_n$ and the phase differences $\phi_n:=\theta_{n+1}-\theta_n$, rather than on the individual phases $\theta_n$.

Focusing our analysis on on- and off-site solutions, restricting the infinite lattice $\Z$ to a finite lattice with off-site boundary conditions at $n=N+1$ (we remark that on-site boundary conditions work also), and writing $f(r,\mu,\epsilon)=\lambda(r,\mu)+\rmi\omega(r,\mu,\epsilon)$ for real-valued nonlinearities $\lambda$ and $\omega$, we see that equation \eqref{e:alg} for the complex amplitudes $z_n$ becomes the system
\begin{equation}\label{e:falg}
0 = \lambda(r,\mu) r_n + \rmi(\omega(r_n,\mu,\varepsilon)-\rho) r_n + \varepsilon c \left( r_{n+1}\rme^{\rmi\phi_n} - 2r_n + r_{n-1}\rme^{-\rmi\phi_{n-1}}\right), \quad 1\leq n\leq N
\end{equation}
with the boundary conditions
\begin{equation}\label{e:fbc}
\begin{array}{rclcl}
(r_0,\phi_0) 		& := & (r_2,-\phi_1) & & \mbox{for  on-site solutions at } n=0 \\
(r_0,\phi_0) 		& := & (r_1,0)       & & \mbox{for off-site solutions at } n=0 \\
(r_{N+1},\phi_N)	& := & (r_N,0)       & & \mbox{off-site boundary condition at } n=N+1
\end{array}
\end{equation}
for the amplitudes $\boldsymbol{r}=(r_n)_{1\leq n\leq N}$ and the phase differences $\boldsymbol{\phi}=(\phi_n)_{1\leq n\leq N-1}$. In particular, the variables we have to solve for are given by
$(\boldsymbol{r},\boldsymbol{\phi},\rho) = ((r_n)_{1\leq n\leq N},(\phi_n)_{1\leq n\leq N-1},\rho) \in X:=\mathbb{R}^N\times\mathbb{R}^{N-1}\times\mathbb{R}$.

%%%%%%%%%%%%%%%%%%%%%%%%%%%%%%%%%%%%%%%%%%%%%%%%%%%%%%%%%%%%%%%%%%%%%%%%%%%%

\subsection{Snaking branches for dissipative coupling}\label{s:2.2}

We set $c=1$ so that \eqref{e:falg} becomes
\begin{equation}\label{e:fdiss}
0 = \lambda(r_n,\mu) r_n + \rmi(\omega(r_n,\mu,\varepsilon)-\rho) r_n + \varepsilon \left( r_{n+1}\rme^{\rmi\phi_n} - 2r_n + r_{n-1}\rme^{-\rmi\phi_{n-1}}\right), \quad 1\leq n\leq N.
\end{equation}
We make the following assumptions on the nonlinearities $\lambda$ and $\omega$, which we illustrated in Figure~\ref{f:synchrony}(i). The first hypothesis on $\lambda$ encodes bistability.

\begin{hyp}\label{h:1}
The function $\lambda\colon\mathbb{R}^2\to\mathbb{R}, (r,\mu)\mapsto\lambda(r,\mu)$ is smooth and satisfies the following:
\begin{compactenum}
\item[Symmetry:] The function $\lambda$ is even in $r$ so that $\lambda(-r,\mu) = \lambda(r,\mu)$.
\item[Equilibria:] For each $\mu\in(0,1)$, the function $\lambda(r,\mu)$ has exactly two nontrivial positive roots, denoted $r = r_\pm(\mu)$, with $0 < r_-(\mu) < r_+(\mu)$.
\item[Stability:] For each $\mu\in(0,1)$ we have $\lambda(0,\mu),\lambda_r(r_+(\mu),\mu) < 0 < \lambda_r(r_-(\mu),\mu)$.
\item[Pitchfork bifurcation:] At $\mu = 0$, the roots $r = 0$ and $r = \pm r_-(\mu)$ of the function $r\lambda(r,\mu)$ collide in a generic subcritical pitchfork bifurcation.
\item[Fold bifurcation:] At $\mu = 1$, the roots $r = r_\pm(\mu)$ collide in a generic saddle-node bifurcation at $r = 1$.
\end{compactenum}
\end{hyp}

The following hypothesis on $\omega$ assumes that the amplitude of oscillations affects their frequency only at higher order in the coupling strength $\varepsilon$.

\begin{hyp}\label{h:2}
There are smooth real-valued functions $\omega_0(\mu)$ and $\omega_2(r,\mu,\varepsilon)$ so that $\omega(r,\mu,\varepsilon)=\omega_0(\mu)+\varepsilon^2\omega_2(r,\mu,\varepsilon)$.
\end{hyp}

To formulate our results on the existence and branch structure of localized synchrony patterns, we first define the curves $\mu_*(s)$, $R_\pm(s)$, and $R_0(s)$ that trace out the parameter $\mu$ and the roots $r_\pm(\mu)$ of the function $\lambda$ as a function of an arclength parameter $s$. The parametrized curves are given by
\[
\mu_*(s) :=
\begin{cases}
s   & 0\leq s\leq 1 \\
2-s & 1\leq s\leq 2,
\end{cases}
\qquad
R_0(s) :=
\begin{cases}
r_-(\mu_*(s)) & 0\leq s\leq 1 \\
r_+(\mu_*(s)) & 1\leq s\leq 2 \\
\end{cases}
\qquad
R_\pm(s) := r_\pm(\mu_*(s)), \quad 0\leq s\leq 2.
\]
and refer to Figure~\ref{f:snaking}(i) for an illustration. Note that the values of $\mu_*(s)$ move back and forth between $\mu=0$ and $\mu=1$ as $s$ increases, while $R_+(s)$ and $R_-(s)$ stay, respectively, on the upper and lower branch of periodic orbits as $s$ varies. Finally, $R_0(s)$ starts at $R_0(0)=0$, traverses the full branch of periodic orbits from $r_-$ to $r_+$ as $s$ increases, and ends at $R_0(2)=r_+(0)$. For each integer $N$ and each $0<\delta\ll1$, we define the connected branch
\begin{eqnarray*}
\Gamma_0^\delta(N) & := &
\left\{ (\boldsymbol{r}, \boldsymbol{\phi}, \rho, \mu) =
\left( (R_0(s), \underbrace{0,\ldots,0}_{2\leq j\leq N}),
\boldsymbol{0}, \omega_0(\mu_*(s)), \mu_*(s) \right) \colon \delta\leq s\leq2 \right\} \\ &&
\bigcup_{k=1}^{N-2}
\left\{ (\boldsymbol{r}, \boldsymbol{\phi}, \rho, \mu) =
\left( (\underbrace{R_+(s),\ldots,R_+(s)}_{1\leq j\leq k}, R_0(s), 0,\ldots,0), 
\boldsymbol{0}, \omega_0(\mu_*(s)), \mu_*(s) \right) \colon 0\leq s\leq2 \right\} \\ &&
\bigcup 
\left\{ (\boldsymbol{r}, \boldsymbol{\phi}, \rho, \mu) =
\left( (\underbrace{R_+(s),\ldots,R_+(s)}_{1\leq j\leq N-1}, R_0(s)),
\boldsymbol{0}, \omega_0(\mu_*(s)), \mu_*(s) \right) \colon 0\leq s\leq1-\delta \right\}
\end{eqnarray*}
of localized synchrony patterns in $X\times[0,1]$ at $\varepsilon=0$. This branch begins close to the stationary state $\boldsymbol{r}=0$ at $\mu=0$ where all nodes are zero and ends close to the domain-filling fully-oscillatory pattern $\boldsymbol{r}=(r_+(1)\ldots,r_+(1))$ at $\mu=1$ where all nodes are set to the oscillation $r_+(1)=r_-(1)$ at the fold bifurcation $\mu=1$. As the connected set $\Gamma_0^\delta(N)$ is traversed, the number of nodes in the localized synchrony pattern that exhibit oscillations grows from $0$ to $N$, and we refer to this phenomenon as ``snaking''; see Figure~\ref{f:snaking}(ii) for an illustration. Note that $\boldsymbol{\phi}=0$ along the branch so that all nodes in the localized synchrony patterns share the same phase.

\begin{figure}
\centering
\includegraphics[scale=0.9]{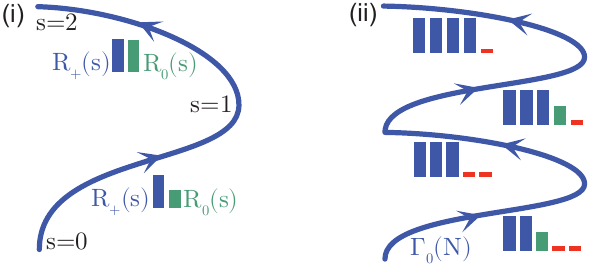}
\caption{Panel~(i) illustrates our parametrization of localized synchrony patterns along the snaking branch $\Gamma_0(N)$. Panel~(ii) indicates how the localized synchrony patterns change along the snaking branch. The oscillatory states along the snaking branch share the same phase so that the phase differences $\phi_n$ vanish.}
\label{f:snaking}
\end{figure}

Our main result establishes the existence of snaking branches $\Gamma_\varepsilon^\delta(N)$ of localized synchrony patterns close to $\Gamma_0^\delta(N)$ for each $N$ for dissipative coupling with sufficiently weak coupling strength $0<\varepsilon\ll1$. For each $\delta>0$ and each subset $\mathcal{S}\subset X\times[0,1]$, we denote by $\mathcal{U}_\delta(\mathcal{S})$ the open $\delta$-neighborhood of $\mathcal{S}$ in $X\times[0,1]$.

\begin{thm}[Snaking for dissipative coupling]\label{t:1}
Assume that $\lambda$ and $\omega$ satisfy Hypotheses~\ref{h:1} and \ref{h:2}, respectively. For each $\delta>0$ there exist constants $\delta_*,\varepsilon_*>0$ so that the following is true for each $\varepsilon\in(0,\varepsilon_*)$:
\begin{compactitem}
\item The set $\mathcal{U}_{\delta_*}(\Gamma_0^\delta(N))$ contains two unique, nonempty, smooth, connected branches $\Gamma_\varepsilon^\mathrm{on}(N)$ and $\Gamma_\varepsilon^\mathrm{off}(N)$ of, respectively, on-site and off-site localized synchrony patterns of \eqref{e:fdiss} with boundary conditions \eqref{e:fbc}.
\item The branches $\Gamma_\varepsilon^\mathrm{on/off}(N)$ depend smoothly on $\varepsilon$ and converge in $C^0$ to $\Gamma_0^\delta(N)$ as $\varepsilon\to0$.
\item The convergence of $\Gamma_\varepsilon^\mathrm{on/off}(N)$ to $\Gamma_0^\delta(N)$ is $C^m$ for each fixed $m\geq1$ outside the set $\mathcal{U}_{\delta_*}(X\times\{0,1\})$, and the branches inside $\mathcal{U}_{\delta_*}(X\times\{0,1\})$ consist of generic fold bifurcations.
\end{compactitem}
\end{thm}

Figure~\ref{f:synchrony} contains numerical continuation results of the localized synchrony patterns for \eqref{e:fdiss} with boundary conditions \eqref{e:fbc} for the nonlinearities
\begin{equation}\label{e:exnonl}
f(r,\mu,\varepsilon) = \lambda(r,\mu) + \rmi \omega(r,\mu,\varepsilon), \qquad
\lambda(r,\mu)=-\mu+2r^2-r^4, \qquad
\omega(r,\mu,\varepsilon)=0,
\end{equation}
which satisfy Hypotheses~\ref{h:1} and~\ref{h:2}, with coupling strength $\varepsilon=0.01$ for $N=10$ nodes.

%%%%%%%%%%%%%%%%%%%%%%%%%%%%%%%%%%%%%%%%%%%%%%%%%%%%%%%%%%%%%%%%%%%%%%%%%%%%

\subsection{Isola branches for conservative coupling}\label{s:2.3}

Next, we discuss conservative coupling with $c=\rmi$ for which \eqref{e:falg} becomes
\begin{equation}\label{e:fcons}
0 = \lambda(r_n,\mu) r_n + \rmi(\omega(r_n,\mu,\varepsilon)-\rho) r_n + \rmi\varepsilon \left( r_{n+1}\rme^{\rmi\phi_n} - 2r_n + r_{n-1}\rme^{-\rmi\phi_{n-1}}\right), \quad 1\leq n\leq N
\end{equation}
and assume that the nonlinearities $\lambda$ and $\omega$ satisfy Hypotheses~\ref{h:1} and~\ref{h:2}. We consider on-site localized synchrony patterns of \eqref{e:fcons} with the boundary conditions \eqref{e:fbc} for $0<\varepsilon\ll1$. For each $k$ with $1\leq k\leq N-2$, we define the isola branch
\begin{eqnarray*}
\Gamma_k & := &
\left\{ (\boldsymbol{r}, \boldsymbol{\phi}, \rho, \mu) = \left(
(\underbrace{R_+(s),\ldots,R_+(s)}_{1\leq j\leq k}, R_0(s),0,0,\ldots),
(\textstyle\underbrace{\textstyle-\frac\pi2,\ldots,-\frac\pi2}_{1\leq j\leq k},\frac\pi2,\ldots),
\omega_0(\mu_*(s)), \mu_*(s) \right) \right\} \\ & \bigcup &
\left\{ (\boldsymbol{r}, \boldsymbol{\phi}, \rho, \mu) = \left(
(\underbrace{R_+(s),\ldots,R_+(s)}_{1\leq j\leq k}, R_0(2-s),R_-(s),0,\ldots),
(\textstyle\underbrace{\textstyle-\frac\pi2,\ldots,-\frac\pi2}_{1\leq j\leq k},\frac\pi2,\ldots),
\omega_0(\mu_*(s)), \mu_*(s) \right) \right\}.
\end{eqnarray*}
We then have the following non-rigorous statement.

\begin{figure}
\centering
\includegraphics[scale=0.9]{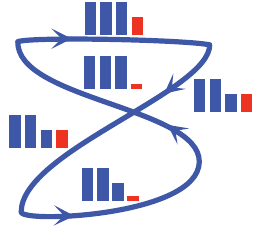}
\caption{We illustrate the isola branch $\Gamma_k$ for $k=2$. The colors indicate the phase difference between adjacent nodes with blue and red corresponding to $\phi_n=-\frac\pi2$ and $\phi_n=\frac\pi2$, respectively. In particular, the phase difference between consecutive nodes along the isola is always $-\frac\pi2$ except for the last node at $n=k+2$ which has phase difference $\frac\pi2$ with the node at $n=k+1$.}
\label{f:isola}
\end{figure}

\begin{stat}\label{t:2}
Assume that $\lambda$ and $\omega$ satisfy Hypotheses~\ref{h:1} and \ref{h:2}, respectively. For each $N$, there exists a constant $\varepsilon_*>0$ so that the following is true for each $\varepsilon\in(0,\varepsilon_*)$. For each $k$ with $1\leq k\leq N-2$, there is a unique smooth isola branch of on-site localized synchrony patterns of \eqref{e:fcons} with boundary conditions \eqref{e:fbc} that is close to the isola branch $\Gamma_k$ and converges in $C^0$ to $\Gamma_k$ as $\varepsilon\to0$ (and in $C^m$ away from $\mu=0,1$).
\end{stat}

Our results are illustrated in Figure~\ref{f:isola}. The results are not rigorous as we will only consider the equations for oscillatory nodes and do not solve the equations for nodes at the stationary state $r=0$.

%%%%%%%%%%%%%%%%%%%%%%%%%%%%%%%%%%%%%%%%%%%%%%%%%%%%%%%%%%%%%%%%%%%%%%%%%%%%

\subsection{Failure of localized synchronization for oscillators with frequency mismatch}\label{s:2.4}

Consider the system \eqref{e:fdiss} with dissipative coupling and assume that $\lambda$ satisfies Hypothesis~\ref{h:1}. Our next result shows localized synchrony patterns that contain nodes set at $r_-$ and at $r_+$ cannot exist if the frequencies of the oscillations at $r_\pm$ differ.

\begin{lem}\label{l:mis}
Consider \eqref{e:fdiss} and assume that $\lambda$ satisfies Hypothesis~\ref{h:1}.
\begin{compactenum}[(i)]
\item Assume that there is a $\mu\in(0,1)$ with $\omega(r_-(\mu),\mu,0)\neq\omega(r_+(\mu),\mu,0)$, then there is a $\varepsilon_*>0$ so that \eqref{e:fdiss} cannot have a solution $\boldsymbol{r}$ for which there are indices $i,j$ with $r_i=r_+(\mu)$ and $r_j=r_-(\mu)$ for any $\varepsilon\in(0,\varepsilon_*)$. The same argument applies to each closed interval $I$ of $(0,1)$ for which $\omega(r_-(\mu),\mu,0)\neq\omega(r_+(\mu),\mu,0)$ uniformly in $\mu\in I$.
\item Assume that $\omega(r,\mu,\varepsilon)=\omega_0(\mu)+\varepsilon\omega_1(r,\mu,\varepsilon)$. If there is a $\mu\in(0,1)$ with
\[
|\omega_1(r_-(\mu),\mu,0)-\omega_1(r_+(\mu),\mu,0)| > \frac{r_+(\mu)}{r_-(\mu)},
\]
then there are constants $k_*\geq0$ and $\varepsilon_*>0$ so that \eqref{e:fdiss} cannot have solutions near
\[
\boldsymbol{r}=(r_+(\mu),\ldots,r_+(\mu),r_-(\mu),0,\ldots,0) \in \R^N
\]
that starts with $k$ contiguous values of $r_+(\mu)$ for any $k\geq k_*$ and any $\varepsilon\in(0,\varepsilon_*)$.
\end{compactenum}
\end{lem}

%%%%%%%%%%%%%%%%%%%%%%%%%%%%%%%%%%%%%%%%%%%%%%%%%%%%%%%%%%%%%%%%%%%%%%%%%%%%

\subsection{Application to a chain of mechanical oscillators}\label{s:2.5}

The authors of \cite{papangelo2017snaking} investigated localized synchronized oscillations in weakly-coupled mechanical oscillators modeled by the second-order system \cite[Equations~(3)-(4)]{papangelo2017snaking}
\begin{equation}\label{e:papa}
\ddot{x}_n + 2\mu\dot{x}_n - \frac{12\pi^2}{5} \dot{x}_n^3 + \frac{16\pi^4}{5} \dot{x}_n^5 + x_n - \varepsilon (x_{n+1}-2x_n+x_{n-1}) = 0,
\quad 1\leq n\leq N,
\end{equation}
where the $N$ nodes are arranged on a ring with periodic boundary conditions, and where $\mu$ is the designated bifurcation parameter. In \cite{papangelo2017snaking}, the authors then substituted the Fourier ansatz $x_n(t)=a_n\cos(\rho t)+b_n\sin(\rho t)$ into \eqref{e:papa} and, upon projecting onto the single harmonic Fourier series, obtained in \cite[Equation~(8)]{papangelo2017snaking} the harmonic-balance model (HBM) of the form \eqref{e:alg} with $z_n=a_n+\rmi b_n$ and $c=\rmi$ for the nonlinearity 
\begin{equation}\label{e:hbm}
f(r,\mu) = -\left(\frac{12\pi^2\rho^5}{8} r^4 - \frac{12\pi^4\rho^3}{5} r^2 + 2\rho\mu \right) + \rmi \left(1-\rho^2\right),
\end{equation}
which satisfies Hypotheses~\ref{h:1} and~\ref{h:2}.

Our non-rigorous results in \S\ref{s:2.3} therefore apply and indicate that the HBM model should exhibit localized synchrony patterns on isolas, and this is consistent with the numerical simulations conducted in \cite[Figures~6, 8, and~9]{papangelo2017snaking}. As indicated in the imaginary part of the nonlinearity in \eqref{e:hbm}, the localized synchrony patterns in the harmonic-balance model necessarily have frequency $\rho=1$ regardless of their amplitude.

\begin{figure}
\centering
\includegraphics[scale=0.9]{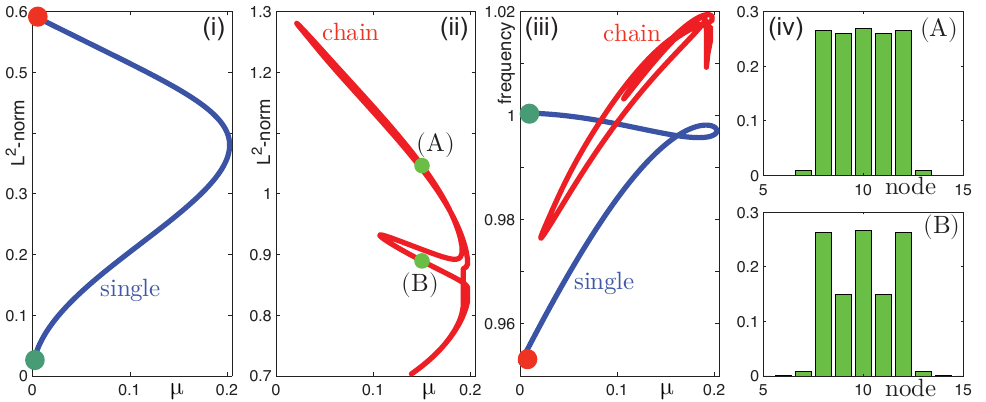}
\caption{Panels~(i) and~(ii) contain numerical continuation results for, respectively, a single oscillator and a chain of $N=20$ oscillators for \eqref{e:papa} with $\varepsilon=0.01$. The frequencies of these solutions are shown in panel~(iii). Panel~(iv) contains two localized synchrony patterns along the branch shown in panel~(ii), and we note that the nodes not shown have $x_n=0$ within computing accuracy.}
\label{f:chain}
\end{figure}

Figure~\ref{f:chain} contains numerical continuation results for single oscillators and localized synchrony patterns for the full model \eqref{e:papa}. As indicated in Figure~\ref{f:chain}(iii), the frequencies of the stable and unstable oscillations are close but not identical. Our results in \S\ref{s:2.4} imply that localized synchronous patterns may not exist on branches that span the interval from Hopf to fold bifurcations, and the continuation results for \eqref{e:papa} shown in Figure~\ref{f:chain}(ii) indeed reveal a very complex isola branch that is not close to the stack of isolas shown in Figure~\ref{f:synchrony}(ii) for the HBM approximation of \eqref{e:papa}.

%%%%%%%%%%%%%%%%%%%%%%%%%%%%%%%%%%%%%%%%%%%%%%%%%%%%%%%%%%%%%%%%%%%%%%%%%%%%
%%%%%%%%%%%%%%%%%%%%%%%%%%%%%%%%%%%%%%%%%%%%%%%%%%%%%%%%%%%%%%%%%%%%%%%%%%%%

\section{Dissipative coupling: proofs of main results}\label{s:3}

We prove Theorem~\ref{t:1} and Lemma~\ref{l:mis} for dissipative coupling. In \S\ref{s:3.1}, we will separate real and imaginary parts of \eqref{e:fdiss} and simplify the resulting system using Hypothesis~\ref{h:2}. Sections~\ref{s:3.2}-\ref{s:3.4} then focus on the proof of Theorem~\ref{t:1}, which we divide into the analysis of, respectively, the bistability region, the fold bifurcation regime near $\mu=1$, and the Hopf bifurcation regime near $\mu=0$. Section~\ref{s:3.5} contains the proof of Lemma~\ref{l:mis}.

%%%%%%%%%%%%%%%%%%%%%%%%%%%%%%%%%%%%%%%%%%%%%%%%%%%%%%%%%%%%%%%%%%%%%%%%%%%%

\subsection{Preparations}\label{s:3.1}

We consider \eqref{e:fdiss}
\[
0 = \lambda(r_n,\mu) r_n + \rmi(\omega(r_n,\mu,\varepsilon)-\rho) r_n + \varepsilon \left( r_{n+1}\rme^{\rmi\phi_n} - 2r_n + r_{n-1}\rme^{-\rmi\phi_{n-1}}\right), \quad 1\leq n\leq N.
\]
Separating the real and imaginary parts of this equation gives the system
\begin{equation}\label{e:sep}
\begin{array}{rcl}
0 & = & \lambda(r_n,\mu) r_n + \varepsilon \left( r_{n+1}\cos\phi_n - 2r_n + r_{n-1}\cos\phi_{n-1} \right) \\
0 & = & (\omega(r_n,\mu,\varepsilon)-\rho) r_n + \varepsilon \left( r_{n+1}\sin\phi_n - r_{n-1}\sin\phi_{n-1} \right),
\end{array}
\end{equation}
where $1\leq n\leq N$. Hypothesis~\ref{h:2} shows that $\omega(r,\mu,\varepsilon)=\omega_0(\mu)+\varepsilon^2\omega_2(r,\mu,\varepsilon)$. Setting $\rho=\omega_0(\mu)-\varepsilon\Omega$ and substituting the expressions for $\rho$ and $\omega(r,\mu,\varepsilon)$ into \eqref{e:sep}, we arrive at the system
\begin{equation}\label{e:disspol}
\begin{array}{rclll}
0 & = & \lambda(r_n,\mu) r_n + \varepsilon \left( r_{n+1}\cos\phi_n - 2r_n + r_{n-1}\cos\phi_{n-1} \right),
& & 1\leq n\leq N \\
0 & = & (\Omega+\varepsilon \omega_2(r_n,\mu,\varepsilon)) r_n + r_{n+1}\sin\phi_n - r_{n-1}\sin\phi_{n-1},
& & 1\leq n\leq N
\end{array}
\end{equation}
together with the boundary conditions \eqref{e:fbc}.

%%%%%%%%%%%%%%%%%%%%%%%%%%%%%%%%%%%%%%%%%%%%%%%%%%%%%%%%%%%%%%%%%%%%%%%%%%%%

\subsection{Analysis inside the bistability region}\label{s:3.2}

We analyze (\ref{e:disspol}) inside the hyperbolic bistability region, that is, for values of $\mu$ away from $0$ and $1$. For each fixed $1\leq k<N$, we will seek localized solutions of (\ref{e:disspol}) in the form
\begin{equation}\label{in-phase ansatz}
r_n = \left\{
\begin{array}{lcl}
r^0_n(\mu)+\varepsilon \sigma_n & & 1\leq n\leq k\\ \displaystyle
\left[\frac{\varepsilon}{\lambda(0,\mu)}\right]^{n-k}r_k^0(\mu)\sigma_n & & k+1\leq n\leq N,
\end{array} \right.
\end{equation}
where $r^0_n(\mu)$ lies in $\{r_-(\mu),r_+(\mu)\}$ for each $1\leq n\leq k$. In particular, we decompose localized solutions into the \textbf{core} nodes for $1\leq n\leq k$, where they attain the amplitudes $r_\pm(\mu)$ of one of the two nonzero periodic orbits and the \textbf{far-field} nodes for $n>k$, where they are close to the equilibrium $r=0$. Before substituting our ansatz into the Ginzburg--Landau system, we introduce the following definition.

\begin{defn}\label{def:Uniform Order}
Let $\Lambda\subset\Z$ be an index set. We say that a family of functions $(f_n(\boldsymbol{\sigma},\boldsymbol{\phi},\varepsilon))_{n \in \Lambda}$ is \textbf{uniformly of order $\varepsilon$} if the following conditions are met.
\begin{compactitem}
\item There exists a $k\in \mathbb{N}$ such that $f_n$ depends at most on $\varepsilon$, $(\sigma_{n-k},\sigma_{n-k+1},\dots,\sigma_{n+k-1},\sigma_{n+k})$, and $(\phi_{n-k},\phi_{n-k+1},\dots,\phi_{n+k-1},\phi_{n+k})$ for each $n \in \Lambda$.
\item For each $\delta>0$ there exists a $C>0$ such that $|f_n|+|D_{(\boldsymbol{\sigma},\boldsymbol{\phi})}f_n|\leq C\varepsilon$ for all $n\in\Lambda$ whenever $|\sigma|+|\Phi|\leq \delta$.
\end{compactitem}
We shall denote a member of such a family of functions by the shorthand $\mathcal{O}^\mathrm{u}(\varepsilon)$.
\end{defn}

With this definition in hand, substituting \eqref{in-phase ansatz} into (\ref{e:disspol}) and dividing by appropriate powers of $\varepsilon$  yields the amplitude equations
\[
\begin{array}{ll}
0 = (\lambda(r^0_n(\mu),\mu)+r^0_n(\mu)\lambda_r(r^0_n(\mu),\mu))\sigma_n & \\
\qquad+[r^0_{n+1}(\mu)\cos\phi_n+r^0_{n-1}(\mu)\cos\phi_{n-1}-2r^0_n(\mu)] + \mathcal{O}^\mathrm{u}(\varepsilon) & 1 \leq n \leq k-1\\
0 = (\lambda(r^0_k(\mu),\mu)+r^0_k(\mu)\lambda_r(r^0_k(\mu),\mu))\sigma_k+[r^0_{k-1}(\mu)\cos\phi_{k-1}-2r^0_k(\mu)]+\mathcal{O}^\mathrm{u}(\varepsilon) & n=k \\
0 = \sigma_{k+1}+r_k^0(\mu)\cos\phi_k+ \mathcal{O}^\mathrm{u}(\varepsilon) & n=k+1 \\
0 = \sigma_n+\sigma_{n-1}\cos\phi_{n-1}+\mathcal{O}^\mathrm{u}(\varepsilon) & k+2 \leq n \leq N
\end{array}
\]
and the phase equations
\[
\begin{array}{ll}
0 = r^0_n(\mu)\Omega+r^0_{n+1}(\mu)\sin\phi_n-r^0_{n-1}(\mu)\sin\phi_{n-1} + \mathcal{O}^\mathrm{u}(\varepsilon) & 1 \leq n \leq k-1\\
0 = r^0_k(\mu)\Omega-r^0_{k-1}(\mu)\sin\phi_{k-1} + \mathcal{O}^\mathrm{u}(\varepsilon) & n=k\\
0 = -r^0_k\sin\phi_k + \mathcal{O}^\mathrm{u}(\varepsilon) & n=k+1\\
0 = -\sigma_{n-1}\sin\phi_{n-1}+\mathcal{O}^\mathrm{u}(\varepsilon) & k+2 \leq n \leq N.
\end{array}
\]
Instead of distinguishing amplitude and phase equations, we separate the core from the far-field nodes. The core system is given by
\begin{equation}\label{e:cc}
\begin{array}{ll}
0 = (\lambda(r^0_n(\mu),\mu)+r^0_n(\mu)\lambda_r(r^0_n(\mu),\mu))\sigma_n & \\
\qquad+[r^0_{n+1}(\mu)\cos\phi_n+r^0_{n-1}(\mu)\cos\phi_{n-1}-2r^0_n(\mu)] + \mathcal{O}^\mathrm{u}(\varepsilon) & 1\leq n \leq k-1\\
0 = (\lambda(r^0_k(\mu),\mu)+r^0_k(\mu)\lambda_r(r^0_k(\mu),\mu))\sigma_k+
[r^0_{k-1}(\mu)\cos\phi_{k-1}-2r^0_k(\mu)]+\mathcal{O}^\mathrm{u}(\varepsilon) & n=k \\
0 = r^0_n(\mu)\Omega+r^0_{n+1}(\mu)\sin\phi_n-r^0_{n-1}(\mu)\sin\phi_{n-1} + \mathcal{O}^\mathrm{u}(\varepsilon) & 1 \leq n \leq k-1\\
0 = r^0_k(\mu)\Omega-r^0_{k-1}(\mu)\sin\phi_{k-1} + \mathcal{O}^\mathrm{u}(\varepsilon) & n=k,
\end{array}
\end{equation}
and the far-field system is given by
\begin{equation}\label{e:ff}
\begin{array}{ll}
0=\sigma_{k+1}+\cos\phi_k+ \mathcal{O}^\mathrm{u}(\varepsilon) & n=k+1\\
0=\sigma_n+\sigma_{n-1}\cos\phi_{n-1}+\mathcal{O}^\mathrm{u}(\varepsilon) & k+2 \leq n\leq N\\
0=-\sin\phi_k + \mathcal{O}^\mathrm{u}(\varepsilon) & n=k+1\\
0=-\sigma_{n-1}\sin\phi_{n-1}+\mathcal{O}^\mathrm{u}(\varepsilon) & k+2 \leq n \leq N.
\end{array}
\end{equation}
We introduce the core and far-field variables via
\[
\boldsymbol{\sigma}^\mathrm{c} = (\sigma_n)_{1\leq n\leq k}, \qquad
\boldsymbol{\phi}^\mathrm{c}   = (\phi_n)_{1\leq n\leq k-1}, \qquad
\boldsymbol{\sigma}^\mathrm{f} = (\sigma_n)_{k+1\leq n\leq N}, \qquad
\boldsymbol{\phi}^\mathrm{f}   = (\phi_n)_{k\leq n\leq N-1}.
\]
Note that the far-field system depends on $a:=(\sigma_k,\phi_{k-1})$ through the $\mathcal{O}^\mathrm{u}(\varepsilon)$ terms but not on any other variables from the core system. Similarly, the core system depends only on $b:=(\sigma_{k+1},\phi_k)$ from the far-field variables through the $\mathcal{O}^\mathrm{u}(\varepsilon)$ terms. We can therefore write the far-field equation \eqref{e:ff} as $F^\mathrm{f}(\boldsymbol{\sigma}^\mathrm{f},\boldsymbol{\phi}^\mathrm{f},a,\varepsilon)=0$ and the core system \eqref{e:cc} as $F^\mathrm{c}(\boldsymbol{\sigma}^\mathrm{c},\boldsymbol{\phi}^\mathrm{c},\Omega,b,\varepsilon)=0$.

%%%%%%%%%%%%%%%%%%%%%%%%%%%%%%%%%%%%%%%%%%%%%%%%%%%%%%%%%%%%%%%%%%%%%%%%%%%%

\paragraph{Solving the far-field equations.}

The following lemma is based on the implicit function theorem and allows us to solve the far-field system \eqref{e:ff} uniquely as a function of $a=(\sigma_k,\boldsymbol{\phi}_{k-1})$.

\begin{lem}[{\cite[Lemma~4.3.2]{bergland}}]
\label{far-field}
Consider the system
\begin{equation}\label{e:ffs}
\begin{array}{ll}
0=\gamma_1\sigma_{k+1}+\gamma_2\cos\phi_k+ \mathcal{O}^\mathrm{u}(\varepsilon) & n=k+1\\
0=\gamma_3\sigma_n+\gamma_4\sigma_{n-1}\cos\phi_{n-1}+\mathcal{O}^\mathrm{u}(\varepsilon) & k+2 \leq n\leq N\\
0=\gamma_5\sin\phi_k + \mathcal{O}^\mathrm{u}(\varepsilon) & n=k+1\\
0=\gamma_6\sigma_{n-1}\sin\phi_{n-1}+\mathcal{O}^\mathrm{u}(\varepsilon) & k+2 \leq n \leq N
\end{array}
\end{equation}
where $\boldsymbol{\gamma}=(\gamma_1,...,\gamma_6)$ consists of nonzero constants, and the error terms of the equations for index $k+1$ may depend on $a=(\sigma_k,\phi_{k-1})$. For $\varepsilon=0$, \eqref{e:ffs} has the solution $(\boldsymbol{\sigma}_0(\boldsymbol{\gamma}),\boldsymbol{\phi}_0,\varepsilon_0)=(\hat{\sigma}(\boldsymbol{\gamma}),\boldsymbol{0},0)$ for each value of $a$, where
\begin{equation*}
\begin{array}{ll}
\hat{\sigma}_{k+1}=-\frac{\gamma_2}{\gamma_1} & n=k+1\\
\hat{\sigma}_n=-\frac{\gamma_2}{\gamma_1}(-\frac{\gamma_4}{\gamma_3})^{n-k-1} & k+2 \leq n \leq N
\end{array}
\end{equation*}
depends smoothly on $\boldsymbol{\gamma}$. Given an open bounded set $A$ in $\mathbb{R}^2$, there exists an $\varepsilon_0>0$ so that \eqref{e:ffs} has a unique solution $(\boldsymbol{\sigma},\boldsymbol{\phi})$ near $(\boldsymbol{\sigma}_0(\boldsymbol{\gamma}),\boldsymbol{\phi}_0)$ for each $0\leq\varepsilon<\varepsilon_0$, $a\in A$, and $\boldsymbol{\gamma}$, and this solution depends smoothly on $(\varepsilon,a,\boldsymbol{\gamma})$ and is $\mathcal{O}^\mathrm{u}(\varepsilon)$-close to the solution $(\boldsymbol{\sigma}_0(\boldsymbol{\gamma}),\boldsymbol{\phi}_0)$.
\end{lem}

We can apply Lemma~\ref{far-field} to solve \eqref{e:ff} by taking $\gamma=(1,1,1,1,-1,-1)$, which yields a unique solution $(\boldsymbol{\sigma}^\mathrm{f},\boldsymbol{\phi}^\mathrm{f})$ as a smooth function of $(a,\varepsilon)$ with
\[
\sigma_n = (-1)^{n-k} + \mathcal{O}^\mathrm{u}(\varepsilon), \qquad k+1 \leq n \leq N,
\]
where the dependence on $(\sigma_k,\boldsymbol{\phi}_{k-1})$ is hidden in the uniformly small remainders. Substituting this solution into the core equation \eqref{e:cc}, we arrive at the system $F^\mathrm{c}(\boldsymbol{\sigma}^\mathrm{c},\boldsymbol{\phi}^\mathrm{c},\Omega,b(\sigma_k,\boldsymbol{\phi}_{k-1},\varepsilon),\varepsilon)=0$, where $b=(\sigma_{k+1},\phi_k)$ is now a function of the core variables $a=(\sigma_k,\phi_{k-1})$. Since the term $b(\sigma_k,\boldsymbol{\phi}_{k-1},\varepsilon)$ enters only through the $\mathcal{O}^\mathrm{u}(\varepsilon)$ terms, we will, with a slight abuse of notation, write the resulting core system that incorporates the solution of the far-field system also as $F^\mathrm{c}(\boldsymbol{\sigma}^\mathrm{c},\boldsymbol{\phi}^\mathrm{c},\Omega,\varepsilon)=0$. We will solve this system next.

%%%%%%%%%%%%%%%%%%%%%%%%%%%%%%%%%%%%%%%%%%%%%%%%%%%%%%%%%%%%%%%%%%%%%%%%%%%%

\paragraph{Solving the core system.}

Next, we solve the core system \eqref{e:cc}.

\begin{lem}
Pick $1\leq k<N$, then the vector $(\boldsymbol{\sigma}^\mathrm{c},\Omega,\boldsymbol{\phi}^\mathrm{c},\varepsilon)=(\boldsymbol{\hat{\sigma}},0,\boldsymbol{0},0)$ given by
\begin{equation}\label{e:ccsol}
\begin{array}{ll}
\displaystyle \hat{\sigma}_n = \frac{2r^0_n(\mu)-r^0_{n+1}(\mu)-r^0_{n-1}(\mu)}{\lambda(r^0_n(\mu),\mu)+r^0_n(\mu)\lambda_r(r^0_n(\mu),\mu)}, & \quad 1\leq n\leq k-1 \\[5mm]
\displaystyle \hat{\sigma}_k = \frac{2r^0_k(\mu)-r^0_{k-1}(\mu)}{\lambda(r^0_k(\mu),\mu)+r^0_k(\mu)\lambda_r(r^0_k(\mu),\mu)}&
\end{array}
\end{equation}
together with either on-site or off-site boundary conditions from \eqref{e:fbc} satisfies \eqref{e:cc}. Furthermore, the linearization of \eqref{e:disspol}, together with the boundary conditions \eqref{e:fbc}, about this solution is invertible provided that $r^0_n(\mu)\neq 0$ for all $1\leq n\leq k$, and we can therefore continue it smoothly in $\varepsilon$ to nonzero values of $\varepsilon$.
\end{lem}

\begin{proof}
The formula \eqref{e:ccsol} can be derived from \eqref{e:cc} upon setting phase lags $\phi_n$, the frequency $\Omega$, and the parameter $\varepsilon$ to zero. It remains to prove the claims about the linearization, which is of the form
\begin{equation*}
D_{(\boldsymbol{\sigma},\Omega,\boldsymbol{\phi})}F^\mathrm{c}(\boldsymbol{\hat{\sigma}},0,\boldsymbol{0},0)=
\begin{pmatrix}
\textbf{D}(\boldsymbol{r^0}) & \boldsymbol{0} \\
\boldsymbol{0} & \boldsymbol{A}_\mathrm{on/off}(\boldsymbol{r^0})
\end{pmatrix},
\end{equation*}
where \textbf{D} is an invertible diagonal matrix with nonzero diagonal entries given by $\textbf{D}_{nn}(\boldsymbol{r^0})=\lambda(r^0_n(\mu),\mu)+r^0_n(\mu)\lambda_r(r^0_n(\mu),\mu)$, and the matrices $\boldsymbol{A}_\mathrm{on/off}$ for the phases encode the boundary conditions for, respectively, on- and off-site solutions through the $(1,2)$-th entry. We need to show that the matrices $\boldsymbol{A}_\mathrm{on/off}$ are both invertible and we do this by computing its determinants via co-factor expansion and an induction argument. We find that
\begin{equation*}
\det(\boldsymbol{A}_{\mathrm{off}}(\boldsymbol{r^0}))=(-1)^{k+1}\left(\prod_{n=1}^{k-1}r^0_n\right) \sum_{n=0}^k(r^0_n)^2
\end{equation*}
and
\begin{equation*}
\det(\boldsymbol{A}_{\mathrm{on}}(\boldsymbol{r^0}))=(-1)^{k+1}\left( (r_0^0)^2+2\sum_{n=1}^k(r^0_n)^2 \right)\prod_{n=1}^{k-1}r_n^0,
\end{equation*}
and refer to \cite[Lemma~4.3.3]{bergland} for the details of this computation. Both determinants are nonzero since we assumed that $r^0_n(\mu)\neq0$ for each $1\leq n\leq k$.
\end{proof}

%%%%%%%%%%%%%%%%%%%%%%%%%%%%%%%%%%%%%%%%%%%%%%%%%%%%%%%%%%%%%%%%%%%%%%%%%%%%

\subsection{Analysis near the fold bifurcation at $\mu=1$}\label{s:3.3}

Next, we analyze (\ref{e:disspol}) near the rightmost boundary $\mu=1$ of the bistability region where the uncoupled amplitude equation exhibits a fold bifurcation of periodic orbits at $(r,\mu)=(1,1)$ when $\varepsilon=0$. We pick $1\leq k<N$ and are interested in synchrony patterns that, at the fold $\mu=1$, have their first $k$ nodes set to $r=1$ and the remaining $N-k$ nodes set to $r=0$.

We change coordinates near $(r,\mu)$ so that the uncoupled amplitude equation near $(r,\mu)=(1,1)$ is transformed into the normal form
\begin{equation}\label{e:foldnf}
\lambda(r,\mu) r = (1-\mu)-(1-r)^2 + \mathcal{O}(|1-\mu|^2,|1-r|^3).
\end{equation}
Motivated by this expansion, we introduce the ansatz
\begin{equation}\label{e:foldscale}
\varepsilon = \tilde{\varepsilon}\nu^2, \qquad
\mu = 1-\tilde{\mu}\nu^2, \qquad
r_n = \left\{
\begin{array}{lcl}
1 + \nu\sigma_n & & 1\leq n\leq k \\ \displaystyle
\left[ \frac{\tilde{\varepsilon}\nu^2}{\lambda(0,1)}\right]^{n-k}\sigma_n & & k+1 \leq n \leq N.
\end{array} \right.
\end{equation}
Substituting these expressions into (\ref{e:disspol}), dividing out by the highest power of $\nu$ and $\tilde{\varepsilon}$ in each equation, and rearranging the result into core and far-field modes, we obtain the core amplitude system
\begin{equation}\label{e:camp}
\begin{array}{ll}
0 = \tilde{\mu} - \sigma_n^2 + \tilde{\varepsilon} [\cos\phi_n + \cos\phi_{n-1} -2 ] +\mathcal{O}^\mathrm{u}(\nu) & 1 \leq n \leq k-1\\
0 = \tilde{\mu} - \sigma_k^2 + \tilde{\varepsilon} [\cos\phi_{k-1} -2] + \mathcal{O}^\mathrm{u}(\nu) & n=k,
\end{array}
\end{equation}
the core phase system
\begin{equation}\label{e:cphase}
\begin{array}{ll}
0 = \Omega + \sin\phi_n - \sin\phi_{n-1} + \mathcal{O}^\mathrm{u}(\nu) & 1\leq n\leq k-1 \\
0 = \Omega - \sin\phi_{k-1} + \mathcal{O}^\mathrm{u}(\nu) & n=k,
\end{array}
\end{equation}
where $\phi_0:=-\phi_1$ for on-site and $\phi_0:=0$ for off-site solutions, and the far-field system
\begin{equation}\label{e:ff-fold}
\begin{array}{ll}                   
0 = \sigma_{k+1} + \cos\phi_k + \mathcal{O}^\mathrm{u}(\nu) & n=k+1\\
0 = \sigma_n + \sigma_{n-1}\cos\phi_{n-1} + \mathcal{O}^\mathrm{u}(\nu) & k+2 \leq n \leq N\\
0 = -\sin\phi_k + \mathcal{O}^\mathrm{u}(\nu)& n=k+1 \\
0 = -\sigma_{n-1}\sin\phi_{n-1} + \mathcal{O}^\mathrm{u}(\nu) & k+2 \leq n \leq N,
\end{array}
\end{equation}
where the $\mathcal{O}^\mathrm{u}(\nu)$ terms are uniform in $\tilde{\varepsilon}\geq0$. The system \eqref{e:ff-fold} agrees with the far-field equation \eqref{e:ff}. Lemma~\ref{far-field} therefore guarantees that we can solve \eqref{e:ff-fold} uniquely for $(\sigma_n)_{k<n\leq N}$ and $(\phi_n)_{k\leq n<N}$ as smooth functions of $(\sigma_k,\phi_{k-1},\tilde{\varepsilon},\nu)$, and substituting this solution into \eqref{e:camp}-\eqref{e:cphase} does not change the leading-order terms. Next, separately for on- and off-site solutions, we can solve the system \eqref{e:cphase} uniquely for $(\phi_n)_{1\leq n<k}$ and $\Omega$ near $(\phi,\Omega)=0$ as a smooth function of $\nu$. Substituting this solution into the remaining system \eqref{e:camp} and using $(\phi_n)_{k\leq n<N}=\mathcal{O}^\mathrm{u}(\nu)$, we finally arrive at the system
\begin{equation}\label{e:foldbif}
\begin{array}{ll}
0 = \tilde{\mu} - \sigma_n^2 + \mathcal{O}^\mathrm{u}(\nu) & 1 \leq n \leq k-1\\
0 = \tilde{\mu} - \sigma_k^2 - \tilde{\varepsilon} + \mathcal{O}^\mathrm{u}(\nu) & n=k
\end{array}
\end{equation}
for $(\sigma_n)_{1\leq n\leq k}$ and $(\tilde{\mu},\tilde{\varepsilon},\nu)$ with $0\leq\nu\ll1$. We consider \eqref{e:foldbif} separately in the coordinate charts $\tilde{\varepsilon}=1$ and $\tilde{\mu}=2$, and refer to Figure~\ref{f:bifurcations}(i) for an illustration of how roots of \eqref{e:foldbif} disappear and to Figure~\ref{f:bifurcations}(iii) for an illustration of the necessity of using two coordinate charts.

\begin{figure}
\centering
\includegraphics[scale=0.8]{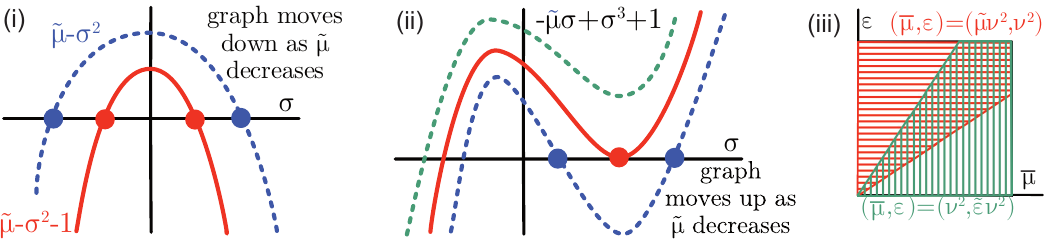}
\caption{Panel~(i) illustrates in which order the roots of the quadratic functions in \eqref{e:foldbif} disappear as $\tilde{\mu}$ decreases. Panel~(ii) indicates how the two positive roots of the cubic function in \eqref{e:finalamp} collide and disappear as $\tilde{\mu}$ decreases. Panel~(iii) indicates how the two charts $\tilde{\varepsilon}=1$ and $\tilde{\mu}=1$ cover the entire relevant $(\mu,\varepsilon)$ region of interest for fold bifurcations, where we set $\bar{\mu}:=1-\mu$ (the case of Hopf bifurcations is similar).}
\label{f:bifurcations}
\end{figure}

%%%%%%%%%%%%%%%%%%%%%%%%%%%%%%%%%%%%%%%%%%%%%%%%%%%%%%%%%%%%%%%%%%%%%%%%%%%%

\paragraph{Solving in the chart $\tilde{\varepsilon}=1$.}

For $\tilde{\varepsilon}=1$, the system \eqref{e:foldbif} is given by
\begin{equation}\label{e:feps}
\begin{array}{ll}
0 = \tilde{\mu} - \sigma_n^2 + \mathcal{O}^\mathrm{u}(\nu) & 1\leq n\leq k-1\\
0 = \tilde{\mu} - \sigma_k^2 - 1 + \mathcal{O}^\mathrm{u}(\nu) & n=k.
\end{array}
\end{equation}
For $\nu=0$, its solutions are given by
\[
\sigma_k := s, \qquad
\tilde{\mu} = 1+s^2, \qquad
(\sigma_n)_{1\leq n<k} = \sqrt{1+s^2} \qquad
\mbox{ where } |s|\leq1,
\]
and we can solve \eqref{e:feps} uniquely for $((\sigma_n)_{1\leq n<k},\tilde{\mu})$ as smooth functions of $\sigma_k=s\in[-1,1]$ and $0\leq\nu\ll1$. A slight extension of this argument shows that the family of solutions we just constructed exhibits a fold bifurcation for $(\tilde{\mu},\sigma_k)=(\tilde{\mu}^\mathrm{fold},\sigma_k^\mathrm{fold})(\nu)=(1+\mathcal{O}^\mathrm{u}(\nu^2),\mathcal{O}^\mathrm{u}(\nu))$. For the original core amplitudes, we showed that there is a $\delta>0$ so that for 
\begin{equation}\label{e:f1}
\mu = 1 - (1+s^2+\mathcal{O}(\sqrt{\varepsilon})) \varepsilon, \quad
r_k = 1 + s \sqrt{1-\mu} + \mathcal{O}(|1-\mu|), \quad
(r_n)_{1\leq n<k} = 1 + \sqrt{1-\mu} + \mathcal{O}(|1-\mu|)
\end{equation}
for $s\in[-1,1]$ and $0\leq\varepsilon\leq\delta$, which exhibits a fold bifurcation for $s=s^\mathrm{fold}(\varepsilon)=\mathcal{O}(\sqrt{\varepsilon})$.

%%%%%%%%%%%%%%%%%%%%%%%%%%%%%%%%%%%%%%%%%%%%%%%%%%%%%%%%%%%%%%%%%%%%%%%%%%%%

\paragraph{Solving in the chart $\tilde{\mu}=2$.}

Setting $\tilde{\mu}=1$, the system \eqref{e:foldbif} becomes
\begin{equation}\label{e:fmu}
\begin{array}{ll}
0 = 2 - \sigma_n^2 + \mathcal{O}^\mathrm{u}(\nu) & 1\leq n\leq k-1\\
0 = 2 - \sigma_k^2 - \tilde{\varepsilon} + \mathcal{O}^\mathrm{u}(\nu) & n=k.
\end{array}
\end{equation}
For $\nu=0$, its solutions are given by
\[
\sigma_k := s, \qquad
\tilde{\varepsilon} = 2-s^2, \qquad
(\sigma_n)_{1\leq n<k} = 1 \qquad
\mbox{ where } |s|\leq2,
\]
and we can solve \eqref{e:fmu} uniquely for $((\sigma_n)_{1\leq n<k},\tilde{\varepsilon})=(1,2-s^2)+\mathcal{O}^\mathrm{u}(\nu)$ as smooth functions of $\sigma_k=s\in[-2,2]$ and $0\leq\nu\ll1$. These results provide a description of the solution branch for each fixed $\mu=1-2\nu^2$ as a function of $\varepsilon=\tilde{\varepsilon}\nu^2=\tilde{\varepsilon}(1-\frac{\mu}{2})$ where $\tilde{\varepsilon}=2-s^2$ for $|s|\leq2$. These solutions intersect $\tilde{\varepsilon}=0$ in exactly two unique points at which $s=\sigma_k=\pm\sqrt{2}+\mathcal{O}^\mathrm{u}(\nu)$. Thus, there is a $\delta>0$ so that the original core amplitudes are given by
\begin{equation}\label{e:f2}
\begin{split}
& \varepsilon = \frac{(2-s^2+\mathcal{O}(\sqrt{1-\mu})) (1-\mu)}{2}, \quad
r_k = 1 + \frac{s \sqrt{2(1-\mu)}}{2} + \mathcal{O}(|1-\mu|), \\
& (r_n)_{1\leq n<k} = 1 + \frac{\sqrt{2(1-\mu)}}{2} + \mathcal{O}(|1-\mu|)
\end{split}
\end{equation}
for $s\in[-2,2]$ and $0\leq1-\mu\leq\delta$, where $\varepsilon=0$ precisely when $s=s^\pm(\mu)=\pm\sqrt{2}+\mathcal{O}(\sqrt{|1-\mu|})$.

%%%%%%%%%%%%%%%%%%%%%%%%%%%%%%%%%%%%%%%%%%%%%%%%%%%%%%%%%%%%%%%%%%%%%%%%%%%%

\paragraph{Combining the results from the two coordinate charts.}

Comparing the expressions in \eqref{e:f1} and \eqref{e:f2} completes the proof of Theorem~\ref{t:1} near $\mu=1$.

%%%%%%%%%%%%%%%%%%%%%%%%%%%%%%%%%%%%%%%%%%%%%%%%%%%%%%%%%%%%%%%%%%%%%%%%%%%%

\subsection{Analysis near the Hopf bifurcation at $\mu=0$}\label{s:3.4}

The final step is to consider the leftmost boundary $\mu=0$ of the bistability region where the rest state $r=0$ undergoes a Hopf bifurcation when $\varepsilon=0$. We pick $k\geq2$ and analyse synchrony patterns that have their first $k-1$ nodes set to $r_+(0)$ and the remaining $N+1-k$ nodes set to $r=0$. We want to show that these synchrony patterns recruit the node at index $k$.

We rescale our variables so that $r_+(0)=1$ and change coordinates so that the uncoupled oscillator equation near $(r,\mu)=(0,0)$ becomes the normal form
\begin{equation}\label{e:hopfnf}
\lambda(r,\mu) r = -\mu r + r^3 + \mathcal{O}(r^2|\mu|,|r|\mu^2,r^4)
\end{equation}
of a subcritical pitchfork bifurcation. We seek solutions in the form
\begin{equation}\label{e:hopfscale}
\varepsilon = \tilde{\varepsilon}\nu^3, \qquad
\mu = \tilde{\mu}\nu^2, \qquad
r_n = \left\{
\begin{array}{lcl}
1+\sigma_n & & 1\leq n\leq k-1 \\
\nu \sigma_k & & n=k \\ \displaystyle
\left[\frac{\tilde{\varepsilon}}{\tilde{\mu}}\right]^{n-k}\nu^{n-k}\sigma_k\sigma_n & & k+1 \leq n \leq N.
\end{array} \right.
\end{equation}
Substituting this ansatz into (\ref{e:disspol}) and dividing out by the highest powers of $\nu$, $\sigma_k$, $\tilde{\varepsilon}$, and $\tilde{\mu}^{-1}$ in each equation, we obtain the core amplitude equations
\[
\begin{array}{ll}
0 = (1+\sigma_n)\lambda_r(1,0)\sigma_n+\mathcal{O}^\mathrm{u}(\nu) & 1\leq n\leq k-1\\
0 = -\tilde{\mu}\sigma_k+\sigma_k^3+\tilde{\varepsilon}(1+\sigma_{k-1})\cos\phi_{k-1}+\mathcal{O}^u(\nu) & n=k,
\end{array}
\]
the core phase equations
\begin{equation}\label{e:cphases0}
\begin{array}{ll}
0=\Omega+[\sin\phi_n-\sin\phi_{n-1}]+\mathcal{O}^u(\nu) & 1\leq n\leq k-2\\
0=\Omega-\sin\phi_{k-2}+\mathcal{O}^u(\nu) & n=k-1 \\
0=-\sin\phi_{k-1}+\mathcal{O}^u(\nu) & n=k,
\end{array}
\end{equation}
and the far-field system
\begin{equation}\label{e:farfields0}
\begin{array}{ll}
0= -\sigma_{k+1}+\cos\phi_k+\mathcal{O}^u(\nu) & n=k+1\\
0= -\sigma_n+\sigma_{n-1}\cos\phi_{n-1}+\mathcal{O}^u(\nu) & k+2\leq n\leq N\\
0=-\sin\phi_k+\mathcal{O}^u(\nu) & n=k+1\\
0=-\sigma_{n-1}\sin\phi_{n-1}+\mathcal{O}^u(\nu) & k+2 \leq n \leq N.\\
\end{array}
\end{equation}
Via Lemma~\ref{far-field}, we can solve system (\ref{e:farfields0}) in terms of $(\phi_{k-1},\tilde{\varepsilon},\tilde{\mu},\nu)$ and substitute these results into our core systems, without altering them at leading order. Separately for on-site and off-site boundary conditions, we can then solve (\ref{e:cphases0}) near $(\phi,\Omega)=\boldsymbol{0}$ as a smooth function of $\nu$. Substituting into the core amplitude equations yields
\begin{equation}\label{e:camps0copy}
\begin{array}{ll}
0 = (1+\sigma_n)\lambda_r(1,0)\sigma_n+\mathcal{O}^\mathrm{u}(\nu) & 1\leq n\leq k-1\\
0=-\tilde{\mu}\sigma_k+\sigma_k^3+\tilde{\varepsilon}(1+\sigma_{k-1})+\mathcal{O}^u(\nu) & n=k.
\end{array}
\end{equation}
The equations for the modes $1\leq n\leq k-1$ can then be solved near $(\sigma_n)_{1\leq n\leq k-1}=\boldsymbol{0}$. Therefore, $\sigma_{k-1}=\mathcal{O}^u(\nu)$ and we can focus our efforts on the equation
\begin{equation}\label{e:finalamp}
0=-\tilde{\mu}\sigma_k+\sigma_k^3+\tilde{\varepsilon}+\mathcal{O}^u(\nu).
\end{equation}
We refer to Figure~\ref{f:bifurcations}(ii) for an illustration of how positive roots of \eqref{e:finalamp} collide and disappear. We shall consider (\ref{e:finalamp}) separately in the charts $\tilde{\varepsilon}=1$ and $\tilde{\mu}=2$.

%%%%%%%%%%%%%%%%%%%%%%%%%%%%%%%%%%%%%%%%%%%%%%%%%%%%%%%%%%%%%%%%%%%%%%%%%%%%

\paragraph{Solving in the chart $\tilde{\varepsilon}=1$.}

In this chart, we obtain
\begin{equation}\label{e:camps0copyeps}
0=-\tilde{\mu}\sigma_k+\sigma_k^3+1+\mathcal{O}^{\mathrm{u}}(\nu)
\end{equation}
Which has $(\sigma_k,\tilde{\mu},\nu)=(s,\frac{1+s^3}{s},0)$ as a solution for any positive $s$. We can thus solve (\ref{e:camps0copyeps}) uniquely for $\tilde{\mu}$ as a smooth function of $s$ and $0\leq\nu\ll1$. An extension of this argument demonstrates that this family of solutions also exhibits a fold bifurcation, at
\begin{equation}
(\sigma_k,\tilde{\mu})=(\tilde{\mu}^{\mathrm{fold}},\sigma^{\mathrm{fold}}_k)(\nu)=\left(\frac{1}{\sqrt[3]{2}},{\frac{3}{2}}\sqrt[3]{2}\right)+\mathcal{O}^{\mathrm{u}}(\nu).
\end{equation}
The methodology for computing this solution is precisely the same as the one found near $\mu=1$. We augment \eqref{e:camps0copyeps} by setting the derivative of the right-hand side with respect to $\sigma_k$ equal to zero, which gives
\begin{eqnarray*}%\label{e:finalampeps_aug}
0 & = & -\tilde{\mu}\sigma_k+\sigma_k^3+1+\mathcal{O}^{\mathrm{u}}(\nu) \\
0 & = & -\tilde{\mu}+3\sigma_k^2+\mathcal{O}^{\mathrm{u}}(\nu).
\end{eqnarray*}
This equation has $(\sigma_k,\tilde{\mu})=\left(\frac{1}{\sqrt[3]{2}},{\frac{3}{2}}\sqrt[3]{2}\right)$ as a solution for $\nu=0$, and we can solve it in a neighborhood of this point for $0\leq\nu\ll 1$.

We also wish to show that the solutions computed in this chart connect with those computed in the chart $\tilde{\mu}=2$. To that end, we must show that there exists two values of $s$ such that $\tilde{\mu}(s,\nu)=2$ for each $\nu\geq 0$ small. In other words, we seek two roots of $g(s,2)+\mathcal{O}^{\mathrm{u}}(\nu)$, where $g(s,\tilde{\mu}):=-\tilde{\mu}+s^3+1$. Explicit calculation shows that $g(s,2)$ has two positive real roots $s=s_\pm$ with
\[
s_-=\frac{\sqrt{5}-1}{2}, \qquad
s_+=1.
\]
Furthermore, we note that $g_s(s_{\pm},2)=-2+3s_{\pm}^2\neq0$, which establishes the claim. Returning to our original coordinates, we see that after shrinking the $\delta>0$ from the previous section if necessary that
\begin{equation}\label{e:somethingelse}
\mu = \left(\frac{1+s^3}{s}+\mathcal{O}(\sqrt[3]{\varepsilon})\right) \varepsilon^{2/3}, \quad
r_k = \left(s+\mathcal{O}(\varepsilon^{1/3})\right)\varepsilon^{1/3}
\end{equation}
for $s\in[\frac{1}{2},\frac{3}{2}]$ and $0\leq\varepsilon\leq \delta$, with a fold bifurcation occurring at $s=s^{\mathrm{fold}}=\frac{1}{\sqrt[3]{2}}+\mathcal{O}(\varepsilon^{1/3})$.

%%%%%%%%%%%%%%%%%%%%%%%%%%%%%%%%%%%%%%%%%%%%%%%%%%%%%%%%%%%%%%%%%%%%%%%%%%%%

\paragraph{Solving in the chart $\tilde{\mu}=2$.}

We now set $\tilde{\mu}=2$, so that (\ref{e:camps0copy}) becomes
\[
0 = -2\sigma_k+\sigma_k^3+\tilde{\varepsilon}+\mathcal{O}^{\mathrm{u}}(\nu)
\]
which has $(\sigma_k,\tilde{\varepsilon},\nu)=(s,s(2-s^2),0)$ as solutions. Thus, we can solve for $\tilde{\varepsilon}$ as a smooth function of $s$ and $0\leq\nu\ll1$. These results represent solutions with fixed $\mu=2\nu^2$ as a function of $\varepsilon=\tilde{\varepsilon}\nu^3=\frac{\tilde{\varepsilon}\mu^{3/2}}{2^{3/2}}$, and they intersect $\tilde{\varepsilon}=0$ in two points: $\sigma_k=0,\sqrt{2}+\mathcal{O}^{\mathrm{u}}(\nu)$.

Therefore, we see that after potentially shrinking the $\delta>0$ found in the section near $\mu=1$,
\begin{equation}\label{e:somethingelse2}
\varepsilon=\frac{\left(s(2-s^2)+\mathcal{O}(\sqrt{\mu})\right)\mu^{3/2}}{2^{3/2}}, \quad r_k=\frac{\left(s+\mathcal{O}(\sqrt{\mu})\right)\sqrt{2\mu}}{2}
\end{equation}
for $s\in[-2,2]$ and $0\leq\mu<\delta$. Finally, we have that $\varepsilon=0$ when $s=0,1+\mathcal{O}(\sqrt{\mu})$.

%%%%%%%%%%%%%%%%%%%%%%%%%%%%%%%%%%%%%%%%%%%%%%%%%%%%%%%%%%%%%%%%%%%%%%%%%%%%

\paragraph{Combining the results from both charts.}

In precisely the same manner as we did near $\mu=1$, comparing \eqref{e:somethingelse} and \eqref{e:somethingelse2} completes the proof of Theorem~\ref{t:1} near $\mu=0$.

%%%%%%%%%%%%%%%%%%%%%%%%%%%%%%%%%%%%%%%%%%%%%%%%%%%%%%%%%%%%%%%%%%%%%%%%%%%%
%%%%%%%%%%%%%%%%%%%%%%%%%%%%%%%%%%%%%%%%%%%%%%%%%%%%%%%%%%%%%%%%%%%%%%%%%%%%

\subsection{Proof of Lemma~\ref{l:mis}}\label{s:3.5}

Recall the equation \eqref{e:sep}
\begin{equation}\label{e:fsep}
\begin{array}{rcl}
0 & = & \lambda(r_n,\mu) r_n + \varepsilon \left( r_{n+1}\cos\phi_n - 2r_n + r_{n-1}\cos\phi_{n-1} \right) \\
0 & = & (\omega(r_n,\mu,\varepsilon)-\rho) r_n + \varepsilon \left( r_{n+1}\sin\phi_n - r_{n-1}\sin\phi_{n-1} \right),
\end{array}
\end{equation}
with $1\leq n\leq N$ that describes synchronous patterns. We focus on $\mu\in(0,1)$.

First, we set $\varepsilon=0$ and obtain
\begin{eqnarray}
0 & = & \lambda(r_n,\mu) r_n \label{e:sepr} \\ \label{e:sepi}
0 & = & (\omega(r_n,\mu,0)-\rho) r_n.
\end{eqnarray}
Any sequence $(r_n)_{1\leq n\leq N}$ with $r_n\in\{0,r_\pm(\mu)\}$ for each $n$ satisfies \eqref{e:sepr}, and these are the only solutions of \eqref{e:sepr}. Assume that there are indices $i$ and $j$ with $r_i=r_+(\mu)$ and $r_j=r_-(\mu)$, then \eqref{e:sepi} for $n=i,j$ is given by
\begin{equation}\label{e:sepf}
\rho = \omega(r_+(\mu),\mu,0), \qquad
\rho = \omega(r_-(\mu),\mu,0)
\end{equation}
after dividing by the nonzero factors $r_\pm(\mu)$. If $\mu\in(0,1)$ is such that $\omega(r_-(\mu),\mu,0)\neq\omega(r_+(\mu),\mu,0)$, then \eqref{e:sepf} cannot have a solution $\rho$. The same argument applies to each closed interval $I$ of $(0,1)$ for which $\omega(r_-(\mu),\mu,0)\neq\omega(r_+(\mu),\mu,0)$ uniformly in $\mu\in I$. This completes the proof of Lemma~\ref{l:mis}(i).

Next, assume that $\omega(r,\mu,\varepsilon)=\omega_0(\mu)+\varepsilon\omega_1(r,\mu,\varepsilon)$, then writing $\rho=\omega_0(\mu)+\varepsilon\Omega$, substituting the expressions for $\rho$ and $\omega(r,\mu,\varepsilon)$ into \eqref{e:fsep}, dividing the phase equations by the common factor $\varepsilon$, and setting $\varepsilon=0$, we obtain
\begin{eqnarray}
0 & = & \lambda(r_n,\mu) r_n \label{e:sepr2} \\ \label{e:sepi2}
0 & = & (\omega_1(r_n,\mu,0)-\Omega) r_n + r_{n+1}\sin\phi_n - r_{n-1}\sin\phi_{n-1}.
\end{eqnarray}
The solutions to \eqref{e:sepr2} are again of the form $r_n\in\{0,r_\pm(\mu)\}$. As in Lemma~\ref{l:mis}(ii), we focus on the solution $\boldsymbol{r}=(r_+(\mu),\ldots,r_+(\mu),r_-(\mu),0,\ldots,0)$ that starts with $k\geq1$ contiguous values of $r_+(\mu)$. Using, for simplicity, the off-site boundary conditions $(r_0,\phi_0)=(r_1,0)$ from \eqref{e:fbc}, the first $k+1$ phase equations are given by
\[
\begin{array}{rclcl}
0 & = & \omega_1(r_+(\mu),\mu,0) - \Omega + \sin\phi_{1} & & n=1 \\
0 & = & \omega_1(r_+(\mu),\mu,0) - \Omega + \sin\phi_n - \sin\phi_{n-1} & & 2\leq n\leq k-1 \\
0 & = & \omega_1(r_+(\mu),\mu,0) - \Omega + \frac{r_-(\mu)}{r_+(\mu)} \sin\phi_k - \sin\phi_{k-1} & & n=k \\
0 & = & \omega_1(r_-(\mu),\mu,0) - \Omega - \frac{r_+(\mu)}{r_-(\mu)} \sin\phi_k & & n=k+1 \\
\end{array}
\]
Adding the equations for $n=1,\ldots,k$, we find that 
\[
(k+1) \left(\omega_1(r_+(\mu),\mu,0) - \Omega \right) + \frac{r_-(\mu)}{r_+(\mu)} \sin\phi_k = 0
\]
In particular, for $k\gg1$, we see that $\Omega=\omega_1(r_+(\mu),\mu,0)+\mathcal{O}(1/k)$ and substituting this solution into the phase equation for $n=k+1$ gives
\[
\sin\phi_k = \frac{r_-(\mu)}{r_+(\mu)} \left(\omega_1(r_-(\mu),\mu,0) - \omega_1(r_+(\mu),\mu,0)\right) + \mathcal{O}(1/k),
\]
which can have a solution uniformly in $k\gg1$  only when $|\omega_1(r_-(\mu),\mu,0)-\omega_1(r_+(\mu),\mu,0)|\leq \frac{r_+(\mu)}{r_-(\mu)}$. This completes the proof of Lemma~\ref{l:mis}(ii).

%%%%%%%%%%%%%%%%%%%%%%%%%%%%%%%%%%%%%%%%%%%%%%%%%%%%%%%%%%%%%%%%%%%%%%%%%%%%
%%%%%%%%%%%%%%%%%%%%%%%%%%%%%%%%%%%%%%%%%%%%%%%%%%%%%%%%%%%%%%%%%%%%%%%%%%%%

\section{Conservative coupling: arguing for isola branches}\label{s:4}

Following the same process as in \S\ref{s:3.1}, we transform \eqref{e:fcons} given by
\[
0 = \lambda(r_n,\mu) r_n + \rmi(\omega(r_n,\mu,\varepsilon)-\rho) r_n + \rmi\varepsilon \left( r_{n+1}\rme^{\rmi\phi_n} - 2r_n + r_{n-1}\rme^{-\rmi\phi_{n-1}}\right), \quad 1\leq n\leq N
\]
into the system
\begin{equation}\label{e:conspol}
\begin{array}{rcl}
0 & = & \lambda(r_n,\mu) r_n + \varepsilon \left( r_{n-1}\sin\phi_{n-1} - r_{n+1}\sin\phi_n \right) \\
0 & = & (\Omega+\varepsilon \omega_2(r_n,\mu,\varepsilon)) r_n + r_{n+1}\cos\phi_n - 2r_n + r_{n-1}\cos\phi_{n-1},
\end{array}
\end{equation}
where $1\leq n\leq N$. Throughout this section, we consider only on-site solutions and focus solely on the leading-order core equations. In particular, we do not prove persistence beyond leading order and do not consider the far-field equations: we believe that the results in \S\ref{s:3} can be applied to make the following results rigorous. The isola branch and the associated localized synchrony patterns, including the phase differences, are shown in Figure~\ref{f:isola_proof}.

\begin{figure}
\centering
\includegraphics[scale=1]{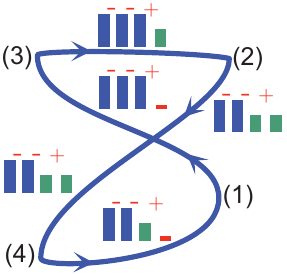}
\caption{Shown are the localized synchrony patterns along an isola branch, where the height of the rectangles indicates whether the amplitude at a node has value $0$, $r_-$, or $r_+$, while the signs shown in red between consecutive nodes in each pattern correspond to the phase difference across these nodes. The numbers at each bifurcation point along the branch will be referenced in the proof.}
\label{f:isola_proof}
\end{figure}

\paragraph{Fold bifurcations near $\mu=1$.} 

Choose $k\geq2$ and assume that $r_n=1$ for $1\leq n\leq k$ and $r_n=0$ for $n\geq k+1$. We use the normal form \eqref{e:foldnf}, apply the scaling \eqref{e:foldscale} with $\tilde{\varepsilon}=1$ to the system \eqref{e:conspol} at $\mu=1$, and set $\nu=0$. The $k$ core phase equations at $\nu=0$ are then given by
\[
\begin{array}{rclcl}
0 & = & \Omega + \cos\phi_n + \cos\phi_{n-1} - 2 & & 1\leq n\leq k-1 \\
0 & = & \Omega + \cos\phi_{k-1} - 2 & & n=k
\end{array}
\]
with $\phi_0:=-\phi_1$, and we select the solutions $\phi_n:=-\frac\pi2$ for $1\leq n\leq k-2$ and $\phi_{k-1}=\frac{\kappa\pi}{2}$ where $\kappa=\pm1$. The $k$ core equations for the amplitudes given by
\[
\begin{array}{rclcl}
0 & = & \tilde{\mu} - \sigma_n^2 + \sin\phi_{n-1} - \sin\phi_n & & 1\leq n\leq k-1 \\
0 & = & \tilde{\mu} - \sigma_k^2 + \sin\phi_{k-1} & & n=k
\end{array}
\]
and substituting the solution for $\phi_n$ we arrive at
\[
\begin{array}{rclcl}
0 & = & \tilde{\mu} - \sigma_1^2 + 2				& & n=1 \\
0 & = & \tilde{\mu} - \sigma_n^2        			& & 1\leq n\leq k-2 \\
0 & = & \tilde{\mu} - \sigma_{k-1}^2 - 1 - \kappa	& & n=k-1 \\
0 & = & \tilde{\mu} - \sigma_k^2 + \kappa			& & n=k.
\end{array}
\]
The node with the lowest value of the constant to the right of the $\sigma_n^2$-term exhibits the first fold bifurcation as $\tilde{\mu}$ is decreased (that is, as $\mu=1$ is approached from below). For $\kappa=-1$ where all nodes have phase differences $-\frac\pi2$, the node with index $k$ therefore exhibits a fold bifurcation as illustrated in Figure~\ref{f:isola_proof}(1). In contrast, when $\kappa=1$ where $\phi_{k-1}=\frac\pi2$ while the remaining nodes have $\phi_n=-\frac\pi2$, it is the node with index $k-1$ that exhibits a fold bifurcation; see Figure~\ref{f:isola_proof}(2). This establishes the transitions near $\mu=1$ shown in Figure~\ref{f:isola_proof}.

\paragraph{Hopf bifurcations near $\mu=0$.}

Choose $k\geq2$ and assume that $r_n=1$ for $1\leq n\leq k-1$ and $r_n=0$ for $n\geq k$. We use the normal form \eqref{e:hopfnf}, apply the scaling \eqref{e:hopfscale} with $\tilde{\varepsilon}=1$ to the system \eqref{e:conspol} at $\mu=0$, and set $\nu=0$. As in \S\ref{s:3.4}, the first $k$ core equations can be solved readily, and we therefore focus on the critical equations for $n=k$ given by
\begin{equation}\label{e:conscrit1}
0 = -\tilde{\mu}\sigma_k + \sigma_k^3 + \sin\phi_{k-1}, \qquad
0 = \cos\phi_{k-1}.
\end{equation}
We see that $\phi_{k-1}=\frac{\kappa\pi}{2}$ with $\kappa=\pm1$ satisfies the second equation, and the first equation becomes $0=-\tilde{\mu}\sigma_k+\sigma_k^3+\kappa$. First, we focus on the case $\kappa=1$, which is the scenario already studied in \S\ref{s:3.4}. As proved in \S\ref{s:3.4}, the node with index $k$ switches from $r_k=0$ to $r_k=r_-(\mu)$ with $\phi_{k-1}=\frac\pi2$. This completes the analysis of the transition in Figure~\ref{f:isola_proof}(3).

Next, we consider the situation where the phase $\phi_{k-1}$ has the opposite sign $\phi_{k-1}=-\frac\pi2$. In this case, the first equation in \eqref{e:conscrit1} becomes $0=-\tilde{\mu}\sigma_k+\sigma_k^3-1$, and the rightmost root $\sigma_k>0$ of this equation cannot undergo a fold bifurcation for $\tilde{\mu}\geq0$. We therefore consider the core amplitude equations for $n\in\{k,k+1\}$ with the scaling
\[
\varepsilon = \nu^3, \qquad
\mu 		= \tilde{\mu}\nu^{8/3}, \qquad
r_k 		= \nu \sigma_k, \qquad
r_{k+1}		= \nu^{4/3} \sigma_{k+1},
\]
where we changed the powers of the scaling for $\mu$ and $\sigma_{k+1}$, and arrive at
\[
0 = \sigma_k^3 + \sin\phi_{k-1}, \qquad
0 = -\tilde{\mu}\sigma_{k+1} + \sigma_{k+1}^3 + \sigma_k \sin\phi_k.
\]
With $\phi_{k-1}=-\frac\pi2$ and $\phi_k=\frac\pi2$, these equations become
\[
0 = \sigma_k^3 - 1, \qquad
0 = -\tilde{\mu}\sigma_{k+1} + \sigma_{k+1}^3 + \sigma_k.
\]
Thus, the first equation has the regular root $\sigma_k$, and the second equation then agrees again with the scenario studied in \S\ref{s:3.4}. We conclude that in this case the node with index $k+1$ switches from $r_{k+1}=0$ to $r_{k+1}=r_-(\mu)$ with $\phi_k=\frac\pi2$. This establishes the transitions near $\mu=0$ shown in Figure~\ref{f:isola_proof}(4).

%%%%%%%%%%%%%%%%%%%%%%%%%%%%%%%%%%%%%%%%%%%%%%%%%%%%%%%%%%%%%%%%%%%%%%%%%%%%
%%%%%%%%%%%%%%%%%%%%%%%%%%%%%%%%%%%%%%%%%%%%%%%%%%%%%%%%%%%%%%%%%%%%%%%%%%%%

\section{Conclusions and outlook}\label{s:5}

\paragraph{Conclusions.} 

We investigated the branch structure of localized synchrony patterns in a finite chain of weakly-coupled oscillators with dissipative and conservative coupling, where each oscillator was modeled by a bistable system with Ginzburg--Landau (or Lambda-Omega) nonlinearity. We showed rigorously that dissipative coupling leads to snaking of on- and off-site localized synchrony patterns where all oscillators share the same phase. We proved that the resulting continuous branch starts close to the state where all nodes are at rest and ends close to the state where all nodes are in their stable oscillatory state. In contrast, we showed formally that conservative coupling leads to a discrete stack of closed isola branches for each fixed number of oscillatory states in the on-site localized synchrony patterns. The phase differences across consecutive nodes is $-\frac\pi2$ except for the last activated node which has phase difference $\frac\pi2$ with the preceding node.

The two main assumptions we made are (i) bistability of the nonlinearity and (ii) frequency match of stable and unstable oscillators at each value of the bifurcation parameter. Bistability is necessary for the existence of localized patterns. If the frequencies of the stable and unstable oscillators do not match for a specific value $\mu=\mu_*$ of the bifurcation parameter, we proved that localized synchrony patterns do not exist for sufficiently weak coupling for $\mu$ near $\mu_*$. In particular, we cannot expect to be able to observe a branch of localized synchrony patterns that crosses through $\mu=\mu_*$.

We applied our results to a chain of weakly-coupled mechanical oscillators. If there is no frequency mismatch between stable and unstable oscillator states, our results imply that dissipative coupling allows us to access a range of localized oscillatory patterns through continuation of a single, fully connected branch. In contrast, for conservative coupling, localized oscillatory patterns lie on distinct disconnected closed isola branches, which makes it impossible to reach all of them through arclength continuation. We also demonstrated through numerical continuation that we do not observe regular isola or snaking branches if the frequencies of stable and unstable oscillator states are different.

Our results may have implications for physical and engineering systems that involve coupled oscillators. Similarly to how localized buckling patterns in thin cylindrical shells can, in principle, be accessed through dead loading starting from a pattern with a single buckling mode, it is possible that localized oscillatory excitation of a small number of nodes may spread to other oscillators and may eventually involve the full oscillator chain. This should be true regardless of whether synchrony patterns lie on stacks of isolas or on snaking branches. Our results indicate that the specific coupling structure (represented by the constant $c$ in \eqref{i:model}) strongly impacts not only the geometric branch structure (through the emergence of either isola or snaking branches) but, equally importantly, also the phase geometry of localized synchrony patterns: the synchrony patterns we found are either all in phase or else their phase differ by $-\frac\pi2$ across neighboring nodes. Depending on the physical or engineering context, this phase behavior may impact the overall dynamics of the underlying structure signicantly.

\paragraph{Open problems.}

We did not prove that the snaking branch begins exactly at the rest state $\boldsymbol{r}=0$ or ends at the fully oscillatory state, though Theorem~\ref{t:1} and the numerical continuation results shown in Figures~\ref{f:synchrony}(iii) and~\ref{f:patterns}(iii) indicate that the branch begins and end very close to these states. We believe that the snaking branch might emerge from the middle branch $\boldsymbol{r}=(r_-(\mu),\ldots,r_-(\mu))$  through $D_N$ symmetry-breaking bifurcations for $\mu$ close to $0$ and $1$, respectively, and note that this conjecture was proved in \cite{tian} for stationary (non-oscillatory) bistable dynamics on lattice rings. We also did not prove the statements for conservative coupling though we are confident that rigorous proofs are possible using the techniques developed here.

An interesting question is whether anti-phase solutions exist in the dissipative coupling regime where the phase differences are $\pi$ instead of $0$. The key difficulty for this case is that it is the first $k-1$ nodes that, to leading order, simultaneously exhibit a fold bifurcation at $\mu=1$ instead of the last activated node at index $k$. Similarly to the case of conservative coupling, we conjecture that anti-phase patterns for dissipative coupling lie on isolas.

Another avenue for future investigation is the stability of the localized oscillations. We expect that the analysis could proceed similarly to the proofs in \cite{bramburger2020stable} and yield exponential convergence to the amplitude components $\boldsymbol{r}$ and significantly slower algebraic convergence for the phase differences $\boldsymbol{\phi}$.

Finally, we focused our analysis solely on the bistable Ginzburg--Landau system \eqref{i:model} with nearest-neighbor coupling for two specific coupling constants that reflect conservative and dissipative coupling. It might be possible to generalize our results near the Hopf and fold bifurcation regimes to coupled systems with general nonlinearities since we exploited only the normal forms of these bifurcations in our analysis. However, we relied heavily on the specific form of how neighboring nodes are coupled both near these bifurcations and also in the bistable regime away from these bifurcations, so the key challenge is to understand how general linear coupling formulations manifest themselves in the normal-form nonlinearities. We believe that our techniques can be extended to analyse localized synchrony patterns in the bistable Ginzburg--Landau system \eqref{i:model} for more general networks beyond nearest-neighbor coupling.

%%%%%%%%%%%%%%%%%%%%%%%%%%%%%%%%%%%%%%%%%%%%%%%%%%%%%%%%%%%%%%%%%%%%%%%%%%%%

\begin{Acknowledgment}
Bergland was supported through the NSF by grant DMS-$2038039$, Bramburger was supported by an NSERC Discovery Grant, and Sandstede was partially supported by the NSF through grant DMS-$2106566$.
\end{Acknowledgment}

%%%%%%%%%%%%%%%%%%%%%%%%%%%%%%%%%%%%%%%%%%%%%%%%%%%%%%%%%%%%%%%%%%%%%%%%%%%%

\bibliographystyle{abbrv}
\bibliography{LocalizedOscillators.bib}

\begin{thebibliography}{10}

\bibitem{aguareles2016asymptotic}
M.~Aguareles, I.~Baldom{\`a}, and T.~M-Seara.
\newblock On the asymptotic wavenumber of spiral waves in systems.
\newblock {\em Nonlinearity}, 30(1):90, 2016.

\bibitem{bergland}
E.~Bergland.
\newblock {\em Multiple Timescales in the Fitzhugh-Nagumo Equation}.
\newblock PhD thesis, Brown University, 2024.

\bibitem{bramburger2019rotating}
J.~J. Bramburger.
\newblock Rotating wave solutions to lattice dynamical systems ii: Persistence
  results.
\newblock {\em Journal of Dynamics and Differential Equations}, 31(1):499--536,
  2019.

\bibitem{bramburger2020stable}
J.~J. Bramburger.
\newblock Stable periodic solutions to lambda-omega lattice dynamical systems.
\newblock {\em Journal of Differential Equations}, 268(7):3201--3254, 2020.

\bibitem{brown2003globally}
E.~Brown, P.~Holmes, and J.~Moehlis.
\newblock Globally coupled oscillator networks.
\newblock In {\em Perspectives and Problems in Nolinear Science}, pages
  183--215. Springer, 2003.

\bibitem{chiu2001synchronization}
C.-H. Chiu, W.-W. Lin, and C.-S. Wang.
\newblock Synchronization in lattices of coupled oscillators with various
  boundary conditions.
\newblock {\em Nonlinear Analysis: Theory, Methods \& Applications},
  46(2):213--229, 2001.

\bibitem{clerc2018chimera}
M.~Clerc, S.~Coulibaly, M.~Ferr{\'e}, and R.~Rojas.
\newblock Chimera states in a duffing oscillators chain coupled to nearest
  neighbors.
\newblock {\em Chaos: An Interdisciplinary Journal of Nonlinear Science},
  28(8):083126, 2018.

\bibitem{cohen1978rotating}
D.~S. Cohen, J.~C. Neu, and R.~R. Rosales.
\newblock Rotating spiral wave solutions of reaction-diffusion equations.
\newblock {\em SIAM Journal on Applied Mathematics}, 35(3):536--547, 1978.

\bibitem{cuevas2019discrete}
J.~Cuevas-Maraver and P.~G. Kevrekidis.
\newblock Discrete breathers in $\phi^4$ and related models.
\newblock In {\em A Dynamical Perspective on the $\phi^4$ Model}, pages
  137--162. Springer, 2019.

\bibitem{ermentrout1990oscillator}
G.~Ermentrout and N.~Kopell.
\newblock Oscillator death in systems of coupled neural oscillators.
\newblock {\em SIAM Journal on Applied Mathematics}, 50(1):125--146, 1990.

\bibitem{ermentrout1991multiple}
G.~B. Ermentrout and N.~Kopell.
\newblock Multiple pulse interactions and averaging in systems of coupled
  neural oscillators.
\newblock {\em Journal of Mathematical Biology}, 29(3):195--217, 1991.

\bibitem{flach1998discrete}
S.~Flach and C.~R. Willis.
\newblock Discrete breathers.
\newblock {\em Physics Reports}, 295(5):181--264, 1998.

\bibitem{fontanela2018dark}
F.~Fontanela, A.~Grolet, L.~Salles, A.~Chabchoub, and N.~Hoffmann.
\newblock Dark solitons, modulation instability and breathers in a chain of
  weakly nonlinear oscillators with cyclic symmetry.
\newblock {\em Journal of Sound and Vibration}, 413:467--481, 2018.

\bibitem{heagy1994synchronous}
J.~Heagy, T.~Carroll, and L.~Pecora.
\newblock Synchronous chaos in coupled oscillator systems.
\newblock {\em Physical Review E}, 50(3):1874, 1994.

\bibitem{hoppensteadt1997weakly}
F.~C. Hoppensteadt and E.~M. Izhikevich.
\newblock {\em Weakly connected neural networks}, volume 126.
\newblock Springer Science \& Business Media, 1997.

\bibitem{kuramoto1984chemical}
Y.~Kuramoto.
\newblock Chemical turbulence.
\newblock In {\em Chemical oscillations, waves, and turbulence}, pages
  111--140. Springer, 1984.

\bibitem{niedergesass2021experimental}
B.~Niederges{\"a}{\ss}, A.~Papangelo, A.~Grolet, A.~Vizzaccaro, F.~Fontanela,
  L.~Salles, A.~Sievers, and N.~Hoffmann.
\newblock Experimental observations of nonlinear vibration localization in a
  cyclic chain of weakly coupled nonlinear oscillators.
\newblock {\em Journal of Sound and Vibration}, 497:115952, 2021.

\bibitem{papangelo2017snaking}
A.~Papangelo, A.~Grolet, L.~Salles, N.~Hoffmann, and M.~Ciavarella.
\newblock Snaking bifurcations in a self-excited oscillator chain with cyclic
  symmetry.
\newblock {\em Communications in Nonlinear Science and Numerical Simulation},
  44:108--119, 2017.

\bibitem{papangelo2018multiple}
A.~Papangelo, N.~Hoffmann, A.~Grolet, M.~Stender, and M.~Ciavarella.
\newblock Multiple spatially localized dynamical states in friction-excited
  oscillator chains.
\newblock {\em Journal of Sound and Vibration}, 417:56--64, 2018.

\bibitem{parker2021standing}
R.~Parker and A.~Aceves.
\newblock Standing-wave solutions in twisted multicore fibers.
\newblock {\em Physical Review A}, 103(5):053505, 2021.

\bibitem{parker2020existence}
R.~Parker, P.~Kevrekidis, and B.~Sandstede.
\newblock Existence and spectral stability of multi-pulses in discrete
  hamiltonian lattice systems.
\newblock {\em Physica D: Nonlinear Phenomena}, 408:132414, 2020.

\bibitem{paullet1994existence}
J.~Paullet, B.~Ermentrout, and W.~Troy.
\newblock The existence of spiral waves in an oscillatory reaction-diffusion
  system.
\newblock {\em SIAM Journal on Applied Mathematics}, 54(5):1386--1401, 1994.

\bibitem{paullet1998spiral}
J.~E. Paullet and G.~B. Ermentrout.
\newblock Spiral waves in spatially discrete $\lambda$-$\omega$ systems.
\newblock {\em International Journal of Bifurcation and Chaos}, 8(01):33--40,
  1998.

\bibitem{shi2015existence}
H.~Shi and Y.~Zhang.
\newblock Existence of breathers for discrete nonlinear {S}chr{\"o}dinger
  equations.
\newblock {\em Applied Mathematics Letters}, 50:111--118, 2015.

\bibitem{shiroky2020nucleation}
I.~Shiroky, A.~Papangelo, N.~Hoffmann, and O.~Gendelman.
\newblock Nucleation and propagation of excitation fronts in self-excited
  systems.
\newblock {\em Physica D: Nonlinear Phenomena}, 401:132176, 2020.

\bibitem{tian}
M.~Tian.
\newblock {\em Patterns in Network Dynamics}.
\newblock PhD thesis, Brown University, 2024.

\bibitem{troy2022logarithmic}
W.~C. Troy.
\newblock Logarithmic spiral solutions of the {K}opell--{H}oward lambda--omega
  reaction--diffusion equations.
\newblock {\em Chaos: An Interdisciplinary Journal of Nonlinear Science},
  32(5):053104, 2022.

\bibitem{vakakis2001normal}
A.~F. Vakakis, L.~I. Manevitch, Y.~V. Mikhlin, V.~N. Pilipchuk, and A.~A.
  Zevin.
\newblock {\em Normal modes and localization in nonlinear systems}.
\newblock Springer, 2001.

\end{thebibliography}

\end{document}